\documentclass[11pt,a4paper,oneside]{article}
\usepackage[english]{babel}
\usepackage[T1]{fontenc}
\usepackage{amsmath}
\usepackage{amsthm}
\usepackage{amsfonts}
\usepackage{amssymb}
\usepackage{lmodern} %Type1-font for non-english texts and characters
\usepackage{indentfirst}		% Indenta o primeiro parÃ¡grafo de cada seÃ§Ã£o.
\usepackage{microtype} 			% para melhorias de justificaÃ§Ã£o
\usepackage[numbers]{natbib}
\usepackage{hyphenat}
\usepackage{pgfplots}  	
\usepackage[latin1]{inputenc}
\usepackage{indentfirst}

\definecolor{blackgreen}{RGB}{0,80,0}

\usepackage[plainpages=false,pdfpagelabels,pdfusetitle,pdfdisplaydoctitle, pdfstartpage=1, pdfstartview={{XYZ null null 1.00}}]{hyperref}
\hypersetup{
	pdftitle={Minimal mass blow-up solutions for a inhomogeneous NLS equation},    
	pdfauthor={Mykael Cardoso and Luiz Gustavo Farah},    
	pdfnewwindow=true,      
	colorlinks=true,      
	linkcolor=blackgreen,         
	citecolor=red}

\usepackage[top=2cm, bottom=2cm, left=2cm, right=2cm]{geometry}
\usepackage{mathtools}
\mathtoolsset{showonlyrefs}

\newcommand{\Real}{\mathbb R}
\newtheorem{thm}{Theorem}[section]
\newtheorem{prop}[thm]{Proposition}
\newtheorem{lemma}[thm]{Lemma}

\newtheorem{coro}[thm]{Corollary}
\numberwithin{equation}{section}

\title{Minimal mass blow-up solutions for a inhomogeneous NLS equation}
\author{Mykael Cardoso and Luiz Gustavo Farah} % Autor
\date{} % Data

\begin{document}
\maketitle
	
\begin{abstract}\noindent
We consider the inhomogeneous nonlinear Schr\"odinger (INLS) equation in $\mathbb{R}^N$
\begin{align}\label{inls}
i \partial_t u +\Delta u +V(x)|u|^{\frac{4-2b}{N}}u = 0,
\end{align}
where $V(x) = k(x)|x|^{-b}$, with $b>0$. Under suitable assumptions on $k(x)$, we established the threshold for global existence and blow-up and then study the existence and non-existence of minimal mass blow-up solutions. 
\end{abstract}

\section{Introduction}
In this paper we consider the initial value problem (IVP) for the inhomogeneous nonlinear Schr\"odinger (INLS) equation
\begin{equation}\label{PVI}
\begin{cases}
i \partial_t u + \Delta u + k(x)|x|^{-b}|u|^{\frac{4-2b}{N}}u = 0, \,\,\, x \in \mathbb{R}^N, \,t>0,\\
u(0) = u_0 \in H^1(\mathbb{R}^N),
\end{cases}
\end{equation}
where $N\geq 1$ and $0<b<\min\{2,N\}$. The specific assumptions on the function $k(x)\in C^1(\mathbb{R}^N)$ will be detailed later. This model is relevant to the investigation of the nonlinear propagation of laser beams influenced by spatially varying interactions, see for instance Belmonte-P{\'e}rez-Vekslerchik-Torres \cite{BBPGVT07}. For $b=0$, the IVP \eqref{PVI} simplifies to the model studied in the classical papers by Merle \cite{Me96} and Rapha{\"e}l-Szeftel \cite{RZ11}. 

When $k(x)\equiv k>0$ is a constant function, the scaling symmetry preserves the $L^2$ norm and the INLS equation is referred to as mass critical for any $0< b<\min\{2,N\}$. In this case, the local well-posedness question for the IVP \eqref{PVI} has been has been extensively investigated in recent literature. We refer the reader to Genoud-Stuart \cite[Appendix K]{GS08} for the proof in the general case $N\geq 1$ and $0<b<\min\{2,N\}$, using the abstract theory developed by Cazenave \cite[Theorem 4.3.1]{Ca03}. Also Guzm\'an \cite[Corollary 1.6]{Ca17}, Dinh \cite[Theorem 1.2]{Di21}, Aloui-Tayachi \cite[Theorem 1.2]{AT21} and Campos-Correia-Farah \cite[Theorem 1.1]{CCF22} obtained similar result, by utilizing Strichartz estimates and dividing the spatial domain into distinct regions (near and far from the origin), reaching the case\footnote{As pointed out in Dinh \cite[Remark 3.2]{Di21} this approach can not be applied to study the case $N=1$.} $N\geq 2$ and $0<b<\min\{2,N/2\}$. Genoud \cite{Ge12} identified sufficient conditions for global well-posedness using a sharp Gagliardo-Nirenberg inequality and we review his result in Section \ref{Not}. Blow-up solutions are also known to exist for this equation. Indeed, if $u$ is a solution of 
\begin{equation}\label{INLSc}
i\partial_{t} u+\Delta u+k|x|^{-b}|u|^{\frac{4-2b}{N}}u=0,
\end{equation}
then, for all $T,\lambda > 0$, from the pseudo-conformal invariance (see Combet-Genoud \cite[Lemma 9]{CG16}) and the scaling symmetry, we deduce that
\begin{equation}\label{QT}
Q_{T, \lambda}\left(x,t\right)=\left(\frac{\lambda}{T-t}\right)^{N / 2} e^{i \lambda^{2} /(T-t){-i|x|^{2} / 4(T-t)}} Q_{k}\left(\frac{\lambda x}{T-t}\right),
\end{equation}
is also a solution of \eqref{INLSc}, where $Q_{k}$ is the unique positive and radial solution of
\begin{equation}\label{GSeq}
-\Delta Q+Q-k|x|^{-b}\left|Q\right|^{\frac{4-2 b}{N}} Q=0, 
\end{equation}
(we refer the reader to Genoud-Stuart \cite{GS08} and Genoud \cite{Ge12} for the existence and uniqueness theory).
Note that $\|Q_{T, \lambda}(t)\|_{L^2}=\|Q_{k}\|_{L^2}$ for all existence time. Combet-Genoud \cite{CG16} establish the classification of such blow-up solutions, following the ideas introduced in the seminal work of Merle \cite{Me93} (see also Campos-Cardoso \cite{CC22} for an alternative proof, based on a compactness theorem in the spirit of Hmidi-Keraani \cite[Theorem 1.1]{HK05}).

The conserved quantities for the flow of \eqref{PVI} are the mass and energy given, respectively, by
\begin{equation}\label{mass}
M[u]= \int|u|^{2}\,dx
\end{equation}
and
\begin{equation}\label{energy}
E[u]=\frac{1}{2} \int|\nabla u|^{2}\,dx-\frac{1}{\frac{4-2 b}{N}+2} \int k(x)|x|^{-b}|u|^{\frac{4-2 b}{N}+2}\, d x.
\end{equation}

Throughout this paper, we assume the following conditions on the function $k(x)\in C^1(\mathbb{R}^N)$.
\begin{itemize}
\item[(\emph{H1})]  $k,\,\, \nabla k\in L^{\infty}(\mathbb{R}^N)$ and $k(0) > 0$;
\item[(\emph{H2})] $x\cdot \nabla k\in L^{\infty}(\mathbb{R}^N)$.
%\item[(\emph{H3})] $\displaystyle k(0)=\max_{x\in \mathbb{R}^N}k(x)$.
\end{itemize} 

As a consequence of (\emph{H1}), the local well-posedness of \eqref{PVI} can be obtained exactly as in \cite{GS08,Ca17,CCF22}. In particular, from the Strichartz estimate approach, the maximal existence time of the solution $T(u_0)>0$ depends only on the $H^1$ norm of the initial data and satisfies $T(u_0)\gtrsim \|u_0\|^{-\alpha}_{H^1}$ for some $\alpha>0$. This, in turn, together with mass conservation imply the blow-up alternative
\begin{equation}\label{BUAlt2}
\text{if}\quad T(u_0)<+\infty,\quad \text{then}\quad  \lim_{t \uparrow T}\|\nabla u(t)\|_{L^2}=+\infty.
\end{equation}
We say that a solution blows up in finite time if that maximal existence time is finite (and therefore the above limit holds).

Our main goal in the present paper is to established the threshold for global existence and blow-up and then study the existence (or not) of blow-up solutions at the threshold level. Unlike the results obtained by Merle \cite{Me96} in the case $b=0$, we prove that the behavior of the function $k(x)$ at the origin is crucial to our analysis and, in particular, our results remain valid even when this function changes sign. Indeed, set $Q_{k(0)}$ to be the unique positive and radial solution of equation \eqref{GSeq} with $k=k(0)$, then as a consequence of a concentration result (see Theorem \ref{Conc1} below), we first obtain the following global result.
\begin{thm}\label{Thm1}
Assume ($\text{H1}$). Let $u_0\in H^1(\mathbb{R}^N)$ such that $\left\|u_{0}\right\|_{L^2}<\|Q_{k(0)}\|_{L^2}$, then the corresponding solution to the IVP \eqref{PVI} is global in $H^{1}$.
\end{thm}
Let $k_1=\|k\|_{L^{\infty}}>0$. If $k(0)<k_1$, then, in view of \eqref{RelQk} below, we have that $\|Q_{k_1}\|_{L^2}< \|Q_{k(0)}\|_{L^2}$. The previous result says that it is possible to find a global solution with mass slightly above the level $\|Q_{k_1}\|_{L^2}$ only assuming (\emph{H1}). This is quite different from the case $b=0$ studied by Merle \cite{Me96}, where the threshold between global existence and blow-up is proved, under an additional condition on $k(x)$, to be $\|Q_{k_1}\|_{L^2}$. 

The next result asserts blow-up in finite time for negative energy solutions. 
\begin{thm}\label{Thm1.1}
Assume ($\text{H1}$), ($\text{H2}$) and 
\begin{equation}\label{k0global}
x \cdot \nabla k(x)\leq 0, \quad \mbox{for all}\quad x\in \mathbb{R}^N.
\end{equation}
If $u_0\in H^1(\mathbb{R}^N)$ is such that $E[u_0]<0$, then the corresponding solution to the IVP \eqref{PVI} blows up in finite time.
\end{thm}
In particular, the assumption \eqref{k0global} implies that the origin is a global maximum, that is, $k(0)=\max_{x\in \mathbb{R}^N}k(x)$. Note that, under this assumption, if $\left\|u_{0}\right\|_{L^2}<\|Q_{k(0)}\|_{L^2}$, then $E[u_0]\geq 0$ (see relations \eqref{Ek-est} and \eqref{EkE} below). Additionally, no radial assumption is needed in the previous theorem and the method of the proof follows the strategy introduced in Cardoso-Farah \cite{CF22}. 
It should be noted that we can not translate the assumption $x \cdot \nabla k(x)\leq 0$ to another point $x_0\neq 0$ as in Merle \cite{Me96}. Indeed, consider the simplest case where the initial data belongs to the virial space $\Sigma=H^1(\mathbb{R}^N)\cap\{f: |x|f\in L^2(\mathbb{R}^N)\}$. For any $x_0\in \mathbb{R}$, the viral identity associated to the solution of \eqref{PVI} reads as follows (see, for instance Farah \cite[Proposition 4.1]{Fa16})
\begin{align}\label{virial}
\frac{d^{2}}{d t^{2}} \int\left|x-x_{0}\right|^{2}|u|^{2} d x&=16 E\left[u_{0}\right]+\frac{4}{\frac{2- b}{N}+1} \int\left(x-x_{0}\right) \cdot \nabla k(x)|x|^{-b}|u|^{\frac{4-2 b}{N}+2}\, d x \nonumber \\
& \quad  +\frac{4 b}{\frac{2- b}{N}+1} \int \frac{x \cdot x_{0}}{|x|^{b+2}} k(x)|u|^{\frac{4-2 b}{N}+2}\, d x.
\end{align}
The first integral in the right hand side of \eqref{virial} is non-positive assuming
\begin{align}\label{cond1}
\left(x-x_{0}\right) \cdot \nabla k(x) \leq 0, \quad \text{for all} \quad x \in \mathbb{R}^{N}.
\end{align}
However, the last integral does not have a definite sign. 
%For this reason, we assume \eqref{cond1} with $x_{0}=0$, so that this integral vanishes.

The next result shows that the level $\|Q_{k(0)}\|_{L^2}$ is the correct threshold for global existence and blow-up. Indeed, we can find blow-up solution slightly above this level assuming a global or local condition on the function $k(x)$.
\begin{thm}\label{Thm2} 
Assume ($\text{H1}$), ($\text{H2}$) and
\begin{itemize}
\item[(i)] $x \cdot \nabla k(x)\leq 0, \quad \text{for all}\quad x\in \mathbb{R}^N$
\end{itemize}
or
\begin{itemize}
\item[(ii)] $x \cdot \nabla k(x)<0, \quad \text {for all}\quad 0<|x|<\ell_{0}$ and some $\ell_{0}>0$.
\end{itemize}
Then, there exists $\varepsilon_{0}>0$ such that for all $\varepsilon \in (0, \varepsilon_{0} )$, there exists $u_{0,\varepsilon} \in H^{1} (\mathbb{R}^{N})$ such that $\left\|u_{0,\varepsilon}\right\|_{L^2}=\|Q_{k(0)}\|_{L^2}+\varepsilon$ and the corresponding solution $u_{\varepsilon}(t)$ of \eqref{PVI} blows up in finite time.
\end{thm}

Under assumption $(i)$ of the previous theorem we in fact have $\varepsilon_{0}=\infty$. Moreover, assumption $(ii)$ implies that the origin is a local strict maximum, that is, $k(0)>k(x)$ for $0<|x|<\rho_{0}$. 

Theorems \ref{Thm1} and \ref{Thm2} naturally lead to the definition of a minimal mass blow-up solution, which is a solution of the IVP \eqref{PVI} with initial data satisfying $\|u_0\|_{L^2}=\|Q_{k(0)}\|_{L^2}$, that blows up in finite time $T>0$. The next result provides a characterization of this type of solution.
\begin{thm}\label{Thm2.1}
Assume ($\text{H1}$). Suppose that $u(t)$ is a minimal mass blow-up solution for the IVP \eqref{PVI} blowing up in finite time $T>0$, then
%\begin{equation}\label{BUAlt}
%\lim_{t \uparrow T}\|\nabla u(t)\|_{L_x^2}=+\infty
%\end{equation}
%and
\begin{equation}\label{WeakLim}
|u(t)|^{2} \rightharpoonup \|Q_{k(0)}\|_{L^2}^{2} \delta_{0}, \quad \mbox{as}\quad t\uparrow T,
\end{equation}
where $\delta_{0}$ stands for the delta function at zero and $T>0$ is the blow-up time.
\end{thm}

Finally, we investigate the problem to find a blow-up solution with minimal mass. For the non-existence we prove
\begin{thm}\label{Thm3}
Assume ({H1}), ({H2}) and that there exist $\ell_{0}, c_0>0$ such that
\begin{equation}\label{Assump-Thm3}
x \cdot \nabla k(x)<-c_0|x|^{1+\alpha_0}, \quad \text {for}\quad |x|<\ell_{0},
\end{equation}
with $\alpha_0 \in (0,1)$. Then there is no minimal mass blow-up solution for the IVP \eqref{PVI}.
%$$
%\left\|u_{0}\right\|_{L^2_x}=\|Q_{k(0)}\|_{L^2_x}
%$$
%and 
%\begin{equation}
%|u(t)|^{2} \rightharpoonup \|Q_{k(0)}\|_{L^2_x}^{2} \delta_{0}, \quad \mbox{as}\quad t \rightarrow T,
%\end{equation}
%where $\delta_{0}$ stands for the delta function at zero and $T>0$ is the blow-up time.
\end{thm}
Note that the assumption \eqref{Assump-Thm3} implies, for some $c>0$, that
\begin{equation}\label{alpha0}
k(0)-k(x)\geq c|x|^{1+\alpha_0}\quad  \text { for }\quad |x|<\ell_0.
\end{equation}
Moreover, it also implies that $k(x)$ is not $C^2(\mathbb{R}^N)$ near the origin.

%Indeed, assuming \eqref{Assump-Thm3} in dimension $N=1$ we have
%\begin{equation}
%-\int_0^x t k'(t)\,dt \geq {c_0}\int _0^x |t|^{1+\alpha_0}\, dt
%\end{equation}
%and integrating by parts
%\begin{equation}
%-t k(t)|_ 0^x +\int_0^x  k(t)\,dt \geq \frac{c_0}{2+\alpha_0}{|x|^{2+\alpha_0}}.
%\end{equation}
%For $x>0$ (the negative case is analogous) we have
%\begin{equation}
%-k(x) +\frac{1}{x}\int_0^x  k(t)\,dt \geq \frac{c_0}{2+\alpha_0}{|x|^{1+\alpha_0}}
%\end{equation}
%and since the origin is a local maximum the inequality \eqref{alpha0} holds. In particular, $\nabla k(0)=0$. If $k(x) \in C^2(\mathbb{R}^N)$, from \eqref{alpha0} and another application of the Mean Value Theorem, there exists $c>0$ and $t\in(0,1)$ such that, for $|x|\ll 1$, we have
%$$
%c|x|^{1+\alpha_0}\leq k(0)-k(x)\leq |x| |\nabla k(tx)|=|x| |\nabla k(tx)-\nabla k(0)|\leq c|x|^2 \|\nabla^2 k\|_{L^{\infty}(|x|\ll 1)},
%$$
%which is an absurd for sufficiently small $|x|$.

On the other hand, assuming a flatness condition on $k(x)$ around the origin, we have the following existence result.
\begin{thm}\label{MinimalBlow}
Assume ({H1})-({H2}) and that there exists $\ell_{0}>0$ such that
\begin{equation}\label{flat}
\quad k(x)=k(0),\quad \text {for}\quad |x|<\ell_{0}.
\end{equation}
Then, for any $T>0$, there exists a minimal mass blow-up solution for the IVP \eqref{PVI} that blows up at time $T$.
%such that 
%$$
%\left\|u_{0}\right\|_{L^2_x}=\|Q_{k(0)}\|_{L^2_x}
%$$
%and 
%\begin{equation}\label{wconc}
%|u(t)|^{2} \rightharpoonup \|Q_{k(0)}\|_{L^2_x}^{2} \delta_{0}, \quad \mbox{as}\quad t \rightarrow T.
%\end{equation}
\end{thm}

It is important to highlight that, in the last two theorems, the analysis primarily requires the behavior of the function $k(x)$ near the origin. The validity of the results persists even if the global maximum of $k(x)$ occurs elsewhere or this function changes sign.

%In the last two theorems only the behavior of the function $k(x)$ close to the origin is needed and even if the global maximum is outside the origin the results remain valid.

To prove Theorem \ref{MinimalBlow}, inspired by Merle \cite{Me90}, we will show that $Q_{T,\lambda}$ defined in \eqref{QT}, solution of equation \eqref{INLSc}, is close to a sequence of solutions to \eqref{PVI} and applying compactness arguments we construct the desired solution. To deal with the potential $|x|^{-b}$ we need a new dispersive estimate in weighted Sobolev spaces (see Lemma \ref{DispEst} below).

The rest of the paper is organized as follows. In Sect. \ref{Not}, we introduce some notations and preliminary estimates. In Sect. \ref{concglo}, we prove some concentration results and use it to deduce the global existence theory and the characterization of minimal mass blow-up solutions stated in Theorems \ref{Thm1} and \ref{Thm2.1}, respectively. Sect. \ref{SecBlow}, is devoted to the existence of blow-up solution and the proof of Theorems \ref{Thm1.1}-\ref{Thm2}. Finally, in Sect. \ref{MMBS}  we study the existence of minimal mass blow-up solutions.

\section{Notations and preliminary results}\label{Not}

The letter $c$ will be used to denote a positive constant that may vary from line to line. If necessary, we use subscript to indicate different parameters the constant may depends on. Given any positive quantities $a$ and $b$, the notation $a\lesssim b$ means that $a\leq cb$, with $c$ uniform with respect to the set where $a$ and $b$ vary. We write $\|\cdot\|_{L^p}$, $p\in [1,+\infty]$, to denote the usual $L^p(\mathbb{R}^N)$ norm in the space variable. If the integration is restricted to a subset $A\subset \mathbb{R}^N$ we use $\|\cdot\|_{L^p(A)}$. To simplify the notation throughout the manuscript we set
$$
\sigma_{b}=\frac{4-2b}{N}+2.
$$
Moreover, the weighted Lebesgue space $L_{b}^{\sigma_{b}}(\mathbb{R}^N)$ is defined to be the normed linear space equipped with the norm  
$$
\|f\|_{L_{b}^{\sigma_{b}}}=\left(\int |x|^{-b}|f(x)|^{\sigma_{b}}\,dx\right)^{1/\sigma_{b}}.
$$

Next, we recall some well known facts about the mass critical INLS equation \eqref{INLSc}. In this case the conserved quantities mass and energy are given, respectively, by
\begin{equation}\label{massk}
M_{k}[u]= \int|u|^{2}\,dx
\end{equation}
and
\begin{equation}\label{energyk}
E_{k}[u]=\frac{1}{2} \int|\nabla u|^{2}\,dx-\frac{k}{\sigma_{b}} \int|x|^{-b}|u|^{\sigma_{b}}\, d x.
\end{equation}
The standing wave solution $u(x,t)=e^{it}Q_k$, where $Q_{k}$ is the unique positive and radial solution of \eqref{GSeq}, plays an important role in the study of the global behavior of the solutions. Indeed, this solution is connected with the best constant for the Gagliardo-Nirenberg type inequality
\begin{equation}\label{GNS}
\int|x|^{-b}|f|^{\sigma_{b}} d x \leq \left(\frac{\sigma_b}{2k\left\|Q_{k}\right\|_{L^2}^{\frac{4-2b}{N}}}\right)\|\nabla f\|_{L^2}^{2}\|f\|_{L^2}^{\frac{4-2 b}{N}},
\end{equation}
see Genoud \cite[Corollary 2.3]{Ge12} and also Farah \cite[Theorem 1.2]{Fa16} for a more general inequality.
Thus, we deduce for all $v\in H^1(\mathbb{R}^N)$ that
\begin{align}\label{Ek-est}
E_{k}[v] & \geq \left(\frac{1}{2} \int|\nabla v|^{2} d x\right)\left(1-\left(\frac{\|v\|_{L^2}}{\left\|Q_{k}\right\|_{L^2}}\right)^{\frac{4-2 b}{N}}\right),
\end{align}
which implies that solutions for the mass critical INLS equation \eqref{INLSc} are global if the initial data satisfies $\|u_0\|_{L^2}< \left\|Q_{k}\right\|_{L^2}$.

Moreover, for all $k>0$ the solutions to \eqref{GSeq} are related according to
\begin{equation}\label{RelQk}
Q_{k}=k^{-\frac{N}{4-2 b}}Q_{1}.
\end{equation}
The Pohozaev identities are given by
$$
\left\|\nabla Q_{k}\right\|_{L^2}^{2}=\frac{N}{2-b}\left\|Q_{k}\right\|_{L^2}^{2}
$$
and
$$
k\|Q_{k}\|^{\sigma_{b}}_{L_{b}^{\sigma_{b}}}=\left(1+\frac{N}{2-b}\right)\|Q_k\|_{L^2}^2,
$$
which implies
\begin{equation}\label{EkQk}
E_k[Q_k]=0.
\end{equation}
Also note that, under the condition $k(0)=\max_{x\in \mathbb{R}^N}k(x)$, definitions \eqref{energy} and \eqref{energyk} imply for all $v\in H^1(\mathbb{R}^N)$
\begin{equation}\label{EkE}
E[v]\geq E_k[v].
\end{equation}
Decay properties of the unique positive and radial solution of equation \eqref{GSeq} were obtained by Genoud-Stuart \cite[Theorem 2.2]{GS08}, in particular, there exists ${\theta}>0$ (independent of $k>0$ by relation \eqref{RelQk}) such that 
\begin{equation}\label{DecayQ0}
\left|Q_{k}(x)\right| \leq c e^{-{\theta}|x|}, \quad \text{for all} \quad x \in \mathbb{R}^{N}.
\end{equation}

Next, we recall the following variational characterization of the ground state
\begin{prop}\label{Prop-VC}
Let $v \in H^{1}(\mathbb{R}^N)$ such that
\begin{equation}\label{VaCh}
\|v\|_{L^2}=\left\|Q_{k}\right\|_{L^2} \quad \text{and} \quad E_{k}[v]=0,
\end{equation}
then there exist $\lambda_{0}>0$ and $\gamma_{0} \in \mathbb{R}$ such that $v(x)=e^{i \gamma_{0}} \lambda_{0}^{N / 2} Q_{k}\left(\lambda_{0} x\right)$
\end{prop}

\begin{proof} In the particular case where $k=1$, this was proved by Combet-Genoud \cite[Proposition 2]{CG16}. For $k >0$, define $v_{k}=k^{\frac{N}{4-2b}} v$, then from assumption \eqref{VaCh} and relation \eqref{RelQk}, we have
$$
\left\|v_{k}\right\|_{L^2}=\left\|Q_{1}\right\|_{L^2} \quad \text{and} \quad E_{1}\left[v_{k}\right]=0.
$$
Therefore, since the conclusion holds for $k=1$, we deduce
$$
v_{k}(x)=e^{i \gamma_{0}} \lambda_{0}^{N / 2} Q_{1}\left(\lambda_{0} x\right) 
$$
thus, the definition of $v_{k}$ and relation \eqref{RelQk} yield
$$
v(x)=e^{i \gamma_{0}} \lambda_{0}^{N / 2} {Q_{k}\left(\lambda_{0} x\right)}.
$$
\end{proof}
A more general statement is the following
\begin{prop}\label{Prop2-VC}
Let $\{v_n\}_{n\in \mathbb N}\subset H^1(\Real^N)$ such that 
\begin{equation}\label{VaCh2}
\left\|v_{n}\right\|_{L^2} \rightarrow\left\|Q_{k}\right\|_{L^2},\quad \left\|\nabla v_{n}\right\|_{L^2} \rightarrow\left\|\nabla Q_{k}\right\|_{L^2} \quad \mbox{and}\quad \displaystyle\limsup _{n } E_{k}\left[v_{n}\right] \leq 0.
\end{equation}
Then, there exists a subsequence of $\{v_n\}$ and $\gamma_0\in \Real$ such that
$$
\lim _{n}\|v_{n}-e^{i \gamma_{0}} Q_{k}\|_{H^{1}}=0.
$$
\end{prop} 
\begin{proof}
This is essentially Proposition 5 in Combet-Genoud \cite{CG16}. Here we give a simplified proof based on the compact embedding $H^1 \hookrightarrow L_{b}^{\sigma_b}$ (see, for instance, Genoud \cite[Lemma 2.1]{Ge12}). Since $\left\{v_n\right\}$ is a bounded sequence in ${H}^{1}$, up to a subsequence, there exists $V \in H^{1}$, such that $v_{n} \rightharpoonup V$ in $H^{1}$ and the compact embedding mentioned above yields 
\begin{equation}\label{CompEmb1}
v_{n} \rightarrow V \quad \text{in} \quad L^{\sigma_b}_b.
\end{equation}
Now, since
$$
E_{k}\left[v_{n}\right]=\frac{1}{2} \int\left|\nabla v_{n}\right|^{2}\, d x-\frac{k}{\sigma_b} \int|x|^{-b}\left|v_{n}\right|^{\sigma_b}\, d x
$$
we deduce, from \eqref{VaCh2}, that $V\neq 0$. Moreover, we apply \eqref{CompEmb1} to obtain
$$
\begin{aligned}
& 0 \geq \limsup_n \frac{1}{2} \int\left|\nabla v_{n}\right|^{2}\, d x-\frac{k}{\sigma_b} \lim_{n} \int|x|^{-b}\left|v_{n}\right|^{\sigma_b}\, d x \\
& \geq  \frac{1}{2} \int|\nabla V|^{2}\, d x-\frac{k}{\sigma_b} \int|x|^{-b}| V|^{\sigma_b} d x=E_{k}[V] \\
& \geq \frac{1}{2} \int | \nabla V |^{2}\, d x\left(1-\left(\frac{\|V\|_{L^2}}{\left\|Q_{k}\right\|_{L^2}}\right)^{\frac{4-2 b}{N}}\right),
\end{aligned}
$$
where in the last inequality we have used the Sharp Gagliardo-Nirenberg inequality \eqref{GNS}. Thus, $\|V\|_{L^2}\geq\left\|Q_{k}\right\|_{L^2}$. On the other hand, since $v_{n} \rightharpoonup V$ in $H^{1}$, assumption \eqref{VaCh2} implies
$$
\|V\|_{L^2} \leq \liminf_{n}\|v_n\|_{L^2}=\| Q_{k} \|_{L^2}\,\,\, \text { and }\,\,\,\|\nabla V\|_{L^2} \leq \liminf_n \left\|\nabla v_{n}\right\|_{L^2}=\left\|\nabla Q_{k}\right\|_{L^2}.
$$
Consequently $\|V\|_{L^2}=\left\|Q_{k}\right\|_{L^2}$ and $E_{k}[V]=0$. From Proposition \ref{Prop-VC}, there exist $\lambda_{0}>0$ and $\gamma_{0}\in \mathbb{R}$ such that
$$
V(x)=e^{i \gamma_{0}} \lambda_{0}^{N / 2} Q_{k}\left(\lambda_{0} x\right) .
$$
Now, again by \eqref{VaCh2}-\eqref{CompEmb1}, we have
$$
\begin{aligned}
0=E_{k}[V] & =\frac{1}{2} \int|\nabla V|^{2}\, d x-\frac{k}{\sigma_b} \int|x|^{-b}| V|^{\sigma_b}\, d x \\
& \leq \frac{1}{2} \int\left|\nabla Q_{k}\right|^{2} \, d x-\frac{k}{\sigma_b} \int|x|^{-b}|V|^{\sigma_b}\, d x \\
& \leq \limsup_{n} E_{k}\left[v_{n}\right] \leq 0,
\end{aligned}
$$
then $\|\nabla V\|_{L^2}=\left\|\nabla Q_{k}\right\|_{L^2}$, which implies $\lambda_{0}=1$.

Finally, since $v_{n} \rightharpoonup V$ in $H^{1}$ and $\left\|v_{n}\right\|_{H^{1}} \rightarrow\|V\|_{H^{1}}$, we must have $v_{n} \rightarrow V=e^{i \gamma_{0}} Q_{k}$ in $H^{1}$ as desired.
\end{proof}

\section{Concentration and global well-posedness}\label{concglo}

From the variational characterization of the ground state given in Proposition \ref{Prop2-VC} we are able to deduce the next concentration result, it shows that any blow up solution must in fact concentrate at least the amount $\|Q_{k(0)}\|_{L^2}$ around the origin due to the potential $|x|^{-b}$. It is worth mentioning that Campos-Cardoso \cite[Theorem 1.1]{CC22} proved an analogous result for the mass critical INLS equation \eqref{INLSc}.
\begin{thm}\label{Conc1} 
Assume ($\text{H1}$). Let $u_0\in H^1(\Real^N)$ such that the corresponding solution to the IVP \eqref{PVI} blows up in finite $T>0$. Then, if $\rho(t)\|\nabla u(t)\|_{L^2} \rightarrow + \infty$ as $t \uparrow T$ we have
$$
\liminf_{t\uparrow T} \int_{|x| \leq \rho(t)}\left|u(t)\right|^{2}\,dx \geq \|Q_{k(0)}\|_{L^2}^{2}.
$$
\end{thm}
\begin{proof}
Our proof again relies on the compact embedding $H^1 \hookrightarrow L_{b}^{\sigma_b}$. Indeed, let $\{t_n\}\subset \Real$ such that $t_n\uparrow T$. For  $\lambda_n=\|\nabla Q_{k(0)}\|_{L^2}/\|\nabla u(t_n)\|_{L^2}$ define 
$$v_n(x)=\lambda_n^{\frac{N}{2}}u(\lambda_nx, t_n).$$
Then,
$$
\|v_n\|_{L^2}=\|u(t_n)\|_{L^2}=\|u_0\|_{L^2}, \quad  \|\nabla v_n\|_{L^2}=\lambda_n\|\nabla u(t_n)\|_{L^2}=\|\nabla Q_{k(0)}\|_{L^2}.
$$
Moreover, since $\lambda_n \rightarrow 0$ we have
\begin{align}
0&=\limsup_{n} \lambda_{n}^{2} E\left[u_{0}\right]=\limsup _{n} \lambda_{n}^{2} E[u(t_{n})] \\
& =\limsup_{n}\left(\frac{1}{2} \int\left|\nabla v_{n}\right|^{2}\, d x-\frac{k(0)}{\sigma_b} \int|x|^{-b}\left|v_{n}\right| ^{\sigma_b}\, d x\right) \\
&\quad  +\lim _{n}\left(\frac{\lambda_{n}^{2}}{\sigma_b} \int(k(0)-k(x)) |x|^{-b}|u(t_{n})|^{\sigma_b}\, d x\right).
\end{align}
We show that the last limit is zero. Indeed, for any $R>0$, we deduce
\begin{align}
	\frac{\lambda_{n}^{2}}{\sigma_b} \int|k(0)-k(x)| |x|^{-b}|u(t_{n})|^{\sigma_b}\, d x &\lesssim \int |k(0)-k(\lambda_nx)||x|^{-b}|v_n|^{\sigma_b}\,dx\\
	&\lesssim\frac{\|k\|_{L^{\infty}}}{R^b}\int_{|x|>R}|v_n|^{\sigma_b}\,dx+\lambda_nR\|\nabla k\|_{L^{\infty}}\int_{|x|<R}|x|^{-b}|v_n|^{\sigma_b}\,dx\\
	&\lesssim\left(\frac{1}{R^{b}}\|k\|_{L^{\infty}}+\lambda_{n}R\|\nabla k\|_{L^{\infty}}  \right)\left\|v_{n}\right\|_{H^{1}}^{\sigma_b},
\end{align} 
by Sobolev embedding and Gagliardo-Nirenberg inequality \eqref{GNS}. Since $\{v_n\}$ is a bounded sequence in $H^1$ and $\lambda_n\to 0$ we deduce the desired limit.

Thus, $\limsup_{n} E_{k(0)}\left[v_n\right]\leq 0$ and arguing as in the proof of Proposition \ref{Prop2-VC} we deduce the existence of $V \in H^{1}$, such that, up to a subsequence, $v_{n} \rightharpoonup V$ in $H^{1}$, $v_{n} \rightarrow V$ in $L^{\sigma_b}_b$ and $\|V\|_{L^2}\geq \|Q_{k(0)}\|_{L^2}$. The rest of the proof follows from classical argument (see, for instance, Hmidi-Keraani \cite[Theorem 2.1]{HK05} and Cardoso-Farah-Guzman \cite[Theorem 1.5]{CFG23}).
\end{proof}

The above theorem immediately yields our global result.
\begin{proof}[Proof of Theorem \ref{Thm1}]	
If the solution $u(t)\in H^1$ for \eqref{PVI} with $u(0)=u_0$ blows up in finite time $T>0$, then, for every $R>0$, mass conservation and the previous theorem imply
$$
\left\|u_{0}\right\|^2_{L^2} \geq \liminf_{t \uparrow T} \int_{|x|<R}\left|u(t)\right|^{2}\, d x \geq\|Q_{k(0)}\|^2_{L^2}>\left\|u_{0}\right\|^2_{L^2},
$$
reaching a contradiction.
\end{proof}

Moreover, Theorem \ref{Conc1} also implies a characterization of the minimal mass blow-up solution.
\begin{proof}[Proof of Theorem \ref{Thm2.1}]
%The limit \eqref{BUAlt} follows directly from the blow-up alternative \eqref{BUAlt2}. To prove the limit \eqref{WeakLim} we first 
First observe that, for any $R>0$, mass conservation and Theorem \ref{Conc1} imply
\begin{align}\label{WeakQ}
\|Q_{k(0)}\|_{L^2}^{2}\geq \limsup_{t \uparrow T} \int_{|x| \leq R}|u(t)|^{2}\, d x  \geq \liminf_{t \uparrow T} \int_{|x| \leq R}|u(t)|^{2}\, d x \geq\|Q_{k(0)}\|_{L^2}^{2}.
\end{align}
and therefore
\begin{align}
\lim_{t \uparrow T}\int_{|x| \leq R}|u(t)|^{2}\, d x=\|Q_{k(0)}\|_{L^2}^{2} \quad \text{and}\quad \lim_{t \uparrow T}\int_{|x| > R}|u(x, t)|^{2}\, d x=0.
\end{align}
As a consequence, we have for all $\varphi \in C_{0}^{\infty}$ that
\begin{align}
\int |u(t)|^{2} \varphi(x)\, d x =\|Q_{k(0)}\|_{L^{2}}^{2}\varphi(0)+\int|u(t)|^{2}(\varphi(x)-\varphi(0))\,d x.
\end{align}
From the continuity of $\varphi$, given any $\varepsilon>0$, there exists $\delta>0$ such that
\begin{align}
\int|u(t)|^{2}|\varphi(x)-\varphi(0)|\, d x & \leq \varepsilon \int_{|x| < \delta}|u(t)|^{2}\, d x+2\|\varphi\|_{L^{\infty}} \int_{|x|>\delta}|u(t)|^{2} \, d x
\end{align}
so, from \eqref{WeakQ}, we have 
$$
\lim_{t \uparrow T}\int|u(t)|^{2}|\varphi(x)-\varphi(0)|\, d x=0
$$
and \eqref{WeakLim} holds.
\end{proof}

We end this section with two useful corollaries. The first one is similar to Proposition 3.2 in Campos-Cardoso \cite{CC22}.
\begin{coro}\label{Cor4}
	Let $u_{n} \in H^{1}(\mathbb{R}^N)$ such that $\left\|u_{n}\right\|_{L^2} \rightarrow \|Q_{k(0)} \|_{L^2}$, $\left\|\nabla u_{n}\right\|_{L^2} \rightarrow \infty$ as $n \rightarrow+\infty$ and $E\left[u_{n}\right] \leq c$. Thus, setting $\lambda_{n}=\|\nabla Q_{k(0)}\|_{L^2}/\|\nabla u_{n}\|_{L^2}$, there exists $\gamma_n \in \mathbb{R}$ such that
\begin{equation}\label{Limun}
\lambda_{n}^{N / 2} e^{i \gamma_{n}} u_{n}\left(\lambda_{n} x\right) \rightarrow Q_{k(0)}(x)\quad  \text { in }\quad  H^{1}(\mathbb{R}^N) 
\end{equation}
and
$$
\left|u_{n}(x)\right|^{2} \rightharpoonup \|Q_{k(0)}\|_{L^2}^{2} \delta_{0}, \quad \text{ as } \quad  n \rightarrow+\infty.
$$
\end{coro}
\begin{proof} Defining $v_{n}(x)=\lambda_{n}^{N / 2} u_{n}\left(\lambda_{n} x\right)$, we have
\begin{align}
	\|v_n\|_{L^2}=\|u_n\|_{L^2}\to \|Q_{k(0)}\|_{L^2}, \quad \|\nabla v_n\|_{L^2}=\|\nabla Q_{k(0)}\|_{L^2}
\end{align}
and, from the proof of Theorem \ref{Conc1},
$$
\limsup_{n}E_{k(0)}[v_n]\leq \limsup_{n}\lambda^2_nE[u_n]\leq \limsup_{n}\lambda^2_n c=0.
$$
Thus, Proposition \ref{Prop2-VC} yields the limit \eqref{Limun}. Moreover, for all $\varphi \in C_0^{\infty}$ we have
\begin{align}
\left|\int\left[\left|u_{n}(x)\right|^{2} \varphi(x)-|Q_{k(0)}(x)|^{2} \varphi(0)\right]\, d x\right| &\leq  \int\left|\left|u_{n}(x)\right|^{2}-\frac{1}{\lambda_{n}^{N}}\left|Q_{k(0)}\left(\frac{x}{\lambda_{n}}\right)\right|^{2}\right| |\varphi(x)|\, d x  \\
& \quad + \left|\int\left[\frac{1}{\lambda^N_{n}}\left|Q_{k(0)}\left(\frac{x}{\lambda_{n}}\right)\right|^2 \varphi(x)-\left|Q_{k(0)}(x)\right|^{2}\varphi(0)\right] \, d x\right|.
\end{align}
Since $\varphi \in L^{\infty}$ and $\left\{u_{n}\right\}$ is a bounded sequence in $L^{2}$, the first term in the right hand side of the last inequality goes to zero as $n\to \infty$ by \eqref{Limun}. On the other hand,
\begin{align}
\int \frac{1}{\lambda_{n}^{N}}\left|Q_{k(0)}\left(\frac{x}{\lambda_{n}}\right)\right|^{2} \varphi(x) \, dx & =\int\left|Q_{k(0)}(x)\right|^{2} \varphi\left(\lambda_{n} x\right) \, d x \\
& =\int\left|Q_{k(0)}(x)\right|^{2}\left(\varphi(\lambda_{n}x)-\varphi(0)\right) d x+\int\left|Q_{k(0)}(x)\right|^{2} \varphi(0)\, d x \\
& \rightarrow \int\left|Q_{k(0)}(x)\right|^{2} \varphi(0) \, d x,  \quad \text{as} \quad n \rightarrow+\infty,
\end{align}
by the Dominated Convergence Theorem, completing the proof.
\end{proof}

Finally, arguing as in the proof of Theorem \ref{Conc1}, we can also obtain the following concentration result.
\begin{coro}\label{Cor2} 
Let $a, b \in \mathbb{R}$ and $\left\{u_{n}\right\} \subset H^{1}(\mathbb{R}^{N})$ such that
$$
\left\|u_{n}\right\|_{L^2} \leq a, \quad E\left[u_{n}\right] \leq b,\quad \left\|\nabla u_{n}\right\|_{L^2} \rightarrow \infty,\quad  \text{as} \quad n \rightarrow+\infty.
$$
Then, for all $R>0$ we have
\begin{align}
	\liminf_{n} \int_{|x|<R}\left|u_n\right|^{2}\, d x\geq \|Q_{k(0)}\|^2_{L^2}.
\end{align}
\end{coro}

\section{Existence of blow-up solutions}\label{SecBlow}
We start this section proving the finite blow-up for negative energy solution
\begin{proof}[Proof of Theorem \ref{Thm1.1}]
The proof is very similar to the one in Cardoso-Farah \cite{CF22}, so we only sketch the main idea. Consider the functions
\begin{align}\label{v(r)}
	v(r)=
	\left\{
	\begin{array}{ll}
		2r, &\mbox{ if } 0\leq r \leq 1\\
		2r-2(r-1)^l, &\mbox{ if } 1< r \leq 1+\left(\frac{1}{l}\right)^{\frac{1}{l-1}}\\
		\mbox{smooth and}\,\, v'<0, &\mbox{ if } 1+\left(\frac{1}{l}\right)^{\frac{1}{l-1}}< r < 2\\
		0,&\mbox{ if }r\geq 2.
	\end{array}
	\right.
\end{align}
and
\begin{align}\label{phi}
\phi(x)=\int_0^{|x|}v(s)\,ds.
\end{align}
Assume by contradiction that the corresponding solution $u(t)$ of \eqref{PVI} exists globally in time and, for $R>0$, define
\begin{align}\label{virial2}
	z_R(t)=\displaystyle\int\phi_R|u(t)|^2\,dx,
\end{align} 
where
$\phi_R(x)=R^2\phi\left(\frac{x}{R}\right)$. Some standard computation imply
\begin{align}\label{zR''}
	z_R''(t)=16E[u_0]+K_{1,R}+K_{2,R}+K_{3,R}+K_{4,R},
\end{align}
where
\begin{align}
	K_{1,R}=&-4\int \left(2-\frac{\partial_r\phi_R}{r}\right)|\nabla u(t)|^2\,dx-4\int \left(\frac{\partial_r \phi_R}{r^3}-\frac{\partial^2_r\phi_R}{r^2}\right)|x\cdot \nabla u(t)|^2\,dx,\\
	K_{2,R}=&\frac{2}{N+2-b}\int\left[(2-b)(2-\partial^2_r\phi_R)+(2N-2+b)\left(2-\frac{\partial_r \phi_R}{r}\right)\right]k(x)|x|^{-b}|u(t)|^{\sigma_b}\,dx,\\
	K_{3,R}=&-\int|u(t)|^2\Delta^2\phi_R\,dx,\\
	K_{4,R}=&\frac{2 N}{N+2-b} \int x\cdot\nabla k(x) \frac{\partial_r \phi_R}{r}|x|^{-b}\left|u(t)\right|^{\sigma_b}\,d x.
\end{align}

Note that $K_{4,R}$ is well defined by assumption (\emph{H2}) and definition \eqref{phi}. Moreover, since $\partial_r \phi(r)=v(r) \geq 0$ and $x \cdot \nabla k(x)\leq 0$, we have that $K_{4,R}\leq 0$. Additionally, since $k\in L^{\infty}$, we can follow the same steps as in the the proof of Theorem 1.1 of \cite{CF22} to control the terms $K_{j,R}$, $j=1,2,3$, and deduce, for $R\gg 1$ and $l>2/b$, that
$$
z_R''(t)\leq 8E[u_0]<0,
$$
reaching a contradiction. 

%Finally, the limit \eqref{LimThm} is a direct consequence of the blow-up alternative \eqref{BUAlt2}.

%Since $k\in L^{\infty}$, we can follow the same steps as in the the proof of Theorem 1.1 of \cite{CF22} to control the terms $K_j$, $j=1,2,3$, and deduce, for $R\gg 1$ and $l>2/b$, that
%\begin{equation}\label{virial3}
%z_R''(t)\leq 8E[u_0]+K_{4,R}.
%\end{equation}
%Finally, in view of $\partial_r \phi(r)=v(r) \geq 0$, if we assume
%$$
%x \cdot \nabla k(x)\leq 0, \quad \mbox{for all} \quad x\in \mathbb{R}^N 
%$$ 
%we have that $K_{4,R}\leq 0$ and the result follows by standard arguments.

\end{proof}

Next, we turn our attention to existence result stated in Theorem \ref{Thm2}. We start with the following preliminary lemma. 
\begin{lemma}\label{Lem1}
For all $\varepsilon \in(0,1)$, there exists a function $u_{0,\varepsilon} \in H^1(\mathbb{R}^{N})$ such that
$$
\left\|u_{0, \varepsilon}\right\|_{L^2}=\|Q_{k(0)}\|_{L^2}+\varepsilon \quad \mbox{and}\quad E\left[u_{0, \varepsilon}\right]\leq -1/\varepsilon.
$$
%\begin{enumerate}
%\item[(a)]  $\left\|u_{0, \varepsilon}\right\|_{L^2}=\|Q_{k(0)}\|_{L^2}+\varepsilon$;
%\item[(b)] $E\left[u_{0, \varepsilon}\right]\leq -A$.
%%\item[(c)] $\left\|\nabla \phi_{\varepsilon}\right\|_{L^2_x} \geq A.$
%\end{enumerate}
\end{lemma}
\begin{proof}
%From the property \eqref{DecayQ0}, we have that $Q_{k(0)}\in L^{\infty}$. 
For $\varepsilon, \lambda>0$ define
$$
f_{\lambda, \varepsilon}(x)=(1+\varepsilon)\lambda^{N/2}Q_{k(0)}(\lambda x).
$$
It is clear that $\|f_{\lambda, \varepsilon}\|_{L^2}=(1+\varepsilon)\|Q_{k(0)}\|_{L^2}$. %and
%$$
%\|\nabla \phi_{\varepsilon, \lambda}\|_{L^2_x}=(1+\varepsilon)\lambda \|\nabla Q_{k(0)}\|_{L^2_x}.
%$$
Moreover
\begin{align}
E\left[f_{\lambda, \varepsilon}\right]& =\frac{1}{2} \int\left|\nabla f_{\lambda, \varepsilon}\right|^{2}\, d x-\frac{1}{\sigma_b} \int k(x)|x|^{-b}\left|f_{\lambda, \varepsilon}\right|^{\sigma_b}\, d x \\
& =E_{k(0)}[f_{\lambda, \varepsilon}]+\frac{1}{\sigma_b} \int (k(0)-k(x))|x|^{-b}|f_{\lambda, \varepsilon}|^{\sigma_b}\, d x \\
& =\lambda^{2}(1+\varepsilon)^{2} E_{k(0)}[Q_{k(0)}]\\
&\quad +\lambda^2((1+\varepsilon)^{2}-(1+\varepsilon)^{\sigma_b}) k(0) \int|x|^{-b}| Q_{k(0)}(x)|^{\sigma_b}\, d x \\
& \quad +\frac{1}{\sigma_b} \int (k(0)-k(x))|x|^{-b}\left|f_{\lambda, \varepsilon}\right|^{\sigma_b}\, d x.
\end{align}
The first term in the right hand side of the last inequality is zero by \eqref{EkQk}. On the other hand, in view of the exponential decay of $Q_{k(0)}$ \eqref{DecayQ0}, we have %Moreover, for any $R>0$, we have
%\begin{align}
%\int|x|^{-b}|\lambda^{N / 2} Q_{k(0)}(\lambda x)|^{\frac{4-2 b}{N}+2} d x 
%&=\lambda^{2} \int|x|^{-b}|Q_{k(0)}(x)|^{\frac{4-2 b}{N}+2} d x\\
%& \leq \lambda^{2}\left(\|Q_{k(0)}\|_{L^\infty}^{\frac{4-2b}{N}+2} \int_{|x| \leq R}|x|^{-b} d x+\frac{1}{R^{b}} \int_{|x| \geq R}\left|Q_{k(0)}(x)\right|^{\frac{4-2 b}{N}+2} d x\right) \\
%& \leq \lambda^{2}\left(c_{1} R^{N-b}+\frac{c_{2}}{R^{b}}\right),
%\end{align}
%where $c_{1}=\|Q_{k(0)}\|_{L^\infty}^{\frac{4-2 b}{N}+2}$ and $c_{2}=\|Q_{k(0)}\|_{L^{\frac{4-2 b}{N}+2}}^{\frac{4-2 b}{N}+2}$. Talking $R>0$ such that $c_{1} R^{N-b}={c_{2}}/{R^{b}}$, there exists $c(\varepsilon)>0$ (independent of $\lambda>0$ ) such that
%$$
%\left((1+\varepsilon)^{2}-(1+\varepsilon)^{\frac{4-2b}{N}+2}\right) k(0) \int|x|^{-b}\left|\lambda^{1 / 2} Q_{k_{2}}\left(\lambda\left(x-x_{0}\right)\right)\right|^{\frac{u-2 b}{N}}+2 d x \leq-\lambda^{2} c(\varepsilon) \text {. }
%$$
\begin{align}
\int (k(0)-k(x))|x|^{-b}\left|f_{\lambda, \varepsilon}\right|^{\sigma_b} d x&\lesssim c \|\nabla k\|_{L^{\infty}}\int |x|^{1-b}\left|f_{\lambda, \varepsilon}\right|^{\sigma_b}\, d x\\
&\lesssim \lambda \|\nabla k\|_{L^{\infty}} \int |x|^{1-b}|Q_{k(0)}|^{\sigma_b}\, d x\\
& \lesssim \lambda,
\end{align}
and
$$
\lambda^2((1+\varepsilon)^{2}-(1+\varepsilon)^{\sigma_b}) k(0) \int|x|^{-b}| Q_{k(0)}( x)|^{\sigma_b}\, d x= c\lambda^{2}((1+\varepsilon)^{2}-(1+\varepsilon)^{\sigma_b}).
$$
Collecting the previous estimates, we obtain
$$
E\left[f_{\lambda, \varepsilon}\right] \leq c\lambda^{2}((1+\varepsilon)^{2}-(1+\varepsilon)^{\sigma_b})+c\lambda.
$$
Now, since $(1+\varepsilon)^{2}<(1+\varepsilon)^{\sigma_b}$, we take $\lambda=\lambda(\varepsilon)$ sufficiently large and set $u_{0, \varepsilon}=f_{\lambda(\varepsilon), \varepsilon}$ to deduce the desired result.
\end{proof}
Now we can proceed to the proof of the first part of Theorem \ref{Thm2}.
\begin{proof}[Proof of Theorem \ref{Thm2}-$(i)$]
This is an immediate consequence of Theorem \ref{Thm1.1} and Lemma \ref{Lem1}.
\end{proof}

Next we study the existence under a local maximum condition. We start with the following refined concentration result.
\begin{lemma} \label{LemConc}
Let $u_{\varepsilon}(t)$ be the solution of \eqref{PVI} with initial data $u_{0,\varepsilon} \in H^1(\mathbb{R}^{N})$ obtained in Lemma \ref{Lem1}. For all $\varepsilon'>0$, there exists $\varepsilon_{0}>0$ such that for all $\varepsilon \in\left(0, \varepsilon_{0}\right)$ and $t>0$, we have
\begin{equation}\label{Conc3}
	\left|\int_{|x|<\varepsilon^{\prime}}|u_{\varepsilon}(t)|^{2}\, d x-\int_{\Real^N}|Q_{k(0)}|^{2}\, d x \right|<\varepsilon^{\prime}.
\end{equation}
and, for every $R>0$
\begin{equation}\label{Conc4}
	\lim_{\varepsilon\to 0}\int_{|x| \geq R}\left|u_{\varepsilon}(t)\right|^{2}\, d x =0, \quad \mbox{uniformly on} \quad t>0.
\end{equation}
\end{lemma}
\begin{proof} Assume, by contradiction, there are $\varepsilon_{n} \rightarrow 0$, $t_{n}>0$ and $\varepsilon'>0$, such that
\begin{equation}\label{Contrad1}
\left|\int_{|x| \leq \varepsilon'}|u_{\varepsilon_{n}}(t_{n})|^{2}\, d x-\int_{\Real^N}|Q_{k(0)}|^{2}\, d x \right| \geq \varepsilon^{\prime}.
\end{equation}
Define $u_{n}(x)=u_{\varepsilon_{n}}(t_{n}, x)$. First observe that
$$
\left\|\nabla u_n\right\|_{L^2} \rightarrow+\infty, \quad \text {as}\quad  n \rightarrow+\infty.
$$
Indeed, if the above limit does not hold, then there exists a subsequence $\varepsilon_{n_k} \rightarrow 0$ such that, by Gagliardo-Nirenberg inequality \eqref{GNS}, we obtain
\begin{align}
|E[u_{0,\varepsilon_{n_k}}]|=\left|E\left[u_{n}\right]\right| \leq \frac{1}{2} \int\left|\nabla u_{n}\right|^{2} \,dx+\|k \|_{L^{\infty}} \int|x|^{-b}\left|u_{n}\right|^{\sigma_b}\, d x \lesssim 1,
\end{align}
which is a contradiction with Lemma \ref{Lem1}. Therefore, the sequence $\{u_n\}$ satisfies the assumptions of Corollary \ref{Cor2} (with $a=2\|Q_{k(0)}\|_{L^2}$ and $b=0$), and thus, for every $R>0$ we have
$$
\liminf_{n }\|u_{n}\|_{L^{2}(B(0,R))} \geq \|Q_{k(0)}\|_{L^2}.
$$
Now, from the conservation of mass and Lemma \ref{Lem1}, we deduce
$$
\|u_{n}\|_{L^2}=\|u_{0,\varepsilon}\|_{L^2} \rightarrow \|Q_{k(0)}\|_{L^2}.
$$
The last two relations imply for every $R>0$ that
$$
\lim_{n} \left(\|u_{n}\|_{L^{2}(B(0,R))}-\|Q_{k(0)}\|_{L^2}\right)=0,
$$ 
reaching a contradiction with \eqref{Contrad1} and therefore \eqref{Conc3} holds.

Next, given $R>0$ and $\varepsilon'<R$ we have
\begin{align}
	\int_{|x| \geq R}|u_{\varepsilon}(t)|^{2}\,dx &\leq \int_{|x| \geq \varepsilon'}|u_{\varepsilon}(t)|^{2}\,dx\\
	&\leq\left|\int_{\mathbb{R}^{N}}| u_{\varepsilon}(t)|^{2}\,d x-\int_{\Real^{N}}|Q_{k(0)}|^{2}\, d x\right|+\left| \int_{|x|<\varepsilon'}\left|u_{\varepsilon}(t)\right|^{2}\,d x-\int_{\mathbb{R}^{N}}| Q_{k(0)}|^{2}\,d x\right|
\end{align}
which implies the limit \eqref{Conc4} in view of \eqref{Conc3} and Lemma \ref{Lem1}.
\end{proof}

The next lemma is also important in our calculations to control the potential energy.
\begin{lemma}\label{GNRef}
For any $A>0$ and $f \in H^{1}(\Real^{N})$, we have
$$
\int_{|x| \geq A}|x|^{-b}|f|^{\sigma_b}\, d x \lesssim\left(\int_{|x| \geq A / 2}|f|^{2}\, d x\right)^{\frac{2-b}{N}}\left(\int_{A \geq|x| \geq A / 2}|f|^{2}\, d x+\int_{|x| \geq A / 2}|\nabla f|^{2}\, d x\right).
$$
\end{lemma}
\begin{proof} Let $\eta \in C^{\infty}$ such that 
$$
\eta(x)=
\left\{
\begin{array}{ll}0, & |x| \leq A/2 \\ 1, & |x| \geq A 
\end{array} \quad \mbox{and} \quad  0 \leq \eta \leq 1.
\right.
$$ 
It is clear that $\eta, \, |\nabla \eta|\in L^{\infty}$. Then, an application of the Gagliardo-Nirenberg inequality \eqref{GNS} yields
$$
\begin{aligned}
	\int_{|x|\geq A}|x|^{-b}|f|^{\sigma_b}\, d x & \leq \int|x|^{-b}|\eta f|^{\sigma_b}\, d x \\
	& \lesssim\left(\int|\eta f|^{2}\, d x\right)^{\frac{2-b}{N}}\left(\int|\nabla(\eta f)|^{2}\, d x\right).
\end{aligned}
$$
The definition of $\eta$ and the product rule imply the desired inequality.
\end{proof}

Next, we proof our existence result assuming a local strict maximum condition at the origin.
\begin{proof}[Proof of Theorem \ref{Thm2}-$(ii)$]
Now, let $\ell_{0}>0$ and assume that
\begin{equation}\label{LocCond12}
x \cdot \nabla k(x)<0, \quad \text {for all}\quad 0<|x|<\ell_{0}.
\end{equation}
Reducing $\ell_{0}$ if necessary, there exists $\bar{c}>0$, such that
\begin{equation}\label{LocCond2}
x \cdot \nabla k(x) \leq -\bar{c}<0,\text{ for all } \ell_{0} / 4<|x|<\ell_{0}<1.
\end{equation}

For any $R>0$ consider the function $\phi_R(x)=R^2\phi({x}/{R})$ and define
\begin{align}\label{virial21}
z_{R}(t)=\displaystyle\int \phi_{R}|u_{\varepsilon}(t)|^{2}\, d x.
\end{align}
Arguing as in the proof of Theorem \ref{Thm1.1}, we obtain, for $R\gg 1$ and $l>2/b$, that 
\begin{align}\label{virial3}
	z''_R(t)\leq 8E[u_{0,\varepsilon}]+K_{4,R}(t),
\end{align}
where 
\begin{align}\label{Pot-term}
	K_{4, R}(t) & =\frac{2 N}{N+2-b} \int x \cdot \nabla k(x) \frac{\partial_{r} \phi_{R}}{r}|x|^{-b}\left|u_{\varepsilon}(t)\right|^{\sigma_b} \, d x \nonumber \\
	& =\frac{2 N}{N+2-b}\left[2 \int_{|x| \leq \ell_{0}} x \cdot \nabla k(x)|x|^{-b}\left|u_{\varepsilon}(t)\right|^{\sigma_b}\, d x\right. \nonumber\\
	&\left.\quad\quad+\int_{|x| \geq \ell_{0}} x \cdot \nabla k(x) \frac{\partial_{r}\phi_{R}}{r}|x|^{-b}\left|u_{\varepsilon}(t)\right|^{\sigma_b}\, d x\right],
\end{align}
and in the last line we have used the definition \eqref{phi}. 

Since $x \cdot \nabla k(x),\,\, {\partial_r \phi_{R}}/{r} \in L^{\infty}$, we obtain from Lemma \ref{GNRef} that
\begin{align}\label{Pot-term2}
&\left|\int_{|x|\geq \ell_{0}}x\cdot\nabla k(x)\frac{\partial_r \phi_R}{r}|x|^{-b}|u_\varepsilon (t)|^{\sigma_b}\,dx\right|\\
& \quad  \leq c\left(\int_{|x|\geq \ell_{0}/2}|u_\varepsilon(t)|^2\,dx\right)^{\frac{2-b}{N}}\left(\int_{\ell_{0} \geq |x| \geq \ell_{0} / 2}|u_\varepsilon(t)|^{2}\, d x+\int_{|x| \geq \ell_{0} / 2}|\nabla u_\varepsilon(t)|^{2}\, d x\right).
\end{align}
Now, we use the viral identity \eqref{zR''} with $R=\ell_0/4$ to control the $L^{2}$-norm of the gradient outside the origin. Indeed, defining $\phi_{0}(x)=({\ell_0}/{4})^{2} \phi({4|x|}/{\ell_0})$ and 
$$
z_{0}(t)=\int \phi_{0}(x)\left|u_{\varepsilon}(t)\right|^{2}\, d x
$$
we have, see \eqref{zR''}, that
\begin{align}
4 \int\left(2-\frac{\partial_r\phi_{0}}{r}\right) |\nabla u_\varepsilon(t)|^2\, dx & =
16 E[u_{0,\varepsilon}]-\int\left|u_{\varepsilon}(t)\right|^{2} \Delta^{2} \phi_{0}\,dx-z_{0}''(t) \\
	& \quad -4\int\left(\frac{\partial_r \phi_{0}}{r^{3}}-\frac{\partial_{r}^{2} \phi_{0}}{r^{2}}\right)\left|x \cdot \nabla u_{\varepsilon}(t)\right|^{2} d x \\
	& \quad +\frac{2}{N+2-b} \int\Phi_0(x) k(x)|x|^{-b}\left|u_{\varepsilon}(t)\right|^{\sigma_b}\,dx \\
	& \quad +\frac{2 N}{N+2-b} \int x \cdot \nabla k(x) \frac{\partial_r\phi_0}{r}|x|^{-b}\left|u_{\varepsilon}(t)\right|^{\sigma_b}\,dx,
\end{align}
where $\Phi_0=(2-b)\left(2-\partial_{r}^{2} \phi_{0}\right)+(2 N-2+b)\left(2-\frac{\partial_r \phi_{0}}{r}\right)$.\\

By construction, we have $\partial_{r} \phi_0(r)-r \partial_{r}^{2}\phi_{0}(r) \geq 0$, for all $r=|x| \in \mathbb{R}$ (see inequality (3.4) in Cardoso-Farah \cite{CF22}), therefore the second line is negative. Moreover $\partial_r\phi_0(r)\geq 0$ and $\text{supp}(\partial_r \phi_0)\subset [0,\ell_0/2]$, so the last line is also negative in view of the local condition \eqref{LocCond12}. In addition, since ${\partial_r \phi_{0}(r)}\leq 2r$, $\text{supp}\left(2-\frac{\partial_r\phi_{0}}{r} \right) \subseteq[\ell_0/4,+\infty)$, $\partial_r\phi_{0}(r)=0$ for $r>\ell_0/2$ and mass conservation, there exists $c>0$ such that
\begin{align}
8 \int_{|x|\geq \ell_0/2} |\nabla u_\varepsilon(t)|^2\,dx &\leq 4 \int\left(2-\frac{\partial_{r}\phi_{0}}{r}\right)|\nabla u_{\varepsilon}(t)|^{2}\,dx\\
& \leq 16 E[u_{0,\varepsilon}] +c-z_{0}''(t) \\
	& \quad +\frac{2}{N+2-b} \int\Phi_0(x) k(x)|x|^{-b}\left|u_{\varepsilon}(t)\right|^{\sigma_b}\, dx.
\end{align}
Since $\text{supp}(\Phi_{0}) \subset [\ell_0/4,+\infty)$, we use the local relation \eqref{LocCond2} and Lemma \ref{GNRef} to estimate the last term as
\begin{align}
	\int\Phi_0(x) k(x)|x|^{-b}\left|u_{\varepsilon}(t)\right|^{\sigma_b}\, dx & \lesssim \int_{\ell_0 / 4\leq |x|\leq \ell_0}|x|^{-b}\left|u_{\varepsilon}(t)\right|^{\sigma_b}\,dx+\int_{|x| \geq \ell_0}|x|^{-b}\left|u_{\varepsilon}(t)\right|^{\sigma_b}\, d x \\
	& \lesssim -\int_{|x| \leq \ell_0} x \cdot \nabla k(x)|x|^{-b}\left|u_{\varepsilon}(t)\right|^{\sigma_b}\, d x\\
	& \quad +c(\varepsilon)\left(\int_{\ell_0 \geq|x| \geq \ell_0 / 2}\left|u_{\varepsilon}(t)\right|^{2}\, d x+\int_{|x| \geq \ell_0 / 2}\left|\nabla u_{\varepsilon}(t)\right|^{2}\, d x\right),
\end{align}
where
$$
c(\varepsilon)=c\left(\int_{|x| \geq \ell_0 / 2}\left|u_{\varepsilon}(t)\right|^{2}\, d x\right)^{\frac{2- b}{N}}.
$$
Therefore, collecting the last two estimates we deduce, from mass conservation, that
\begin{align}\label{Grad-term}
8 \int_{|x|\geq \ell_0/2}\left|\nabla u_{\varepsilon}(t)\right|^{2}\, dx &\leq 16 E[u_{0,\varepsilon}] -z_{0}''(t)+c-c\int_{|x| \leq \ell_0} x \cdot \nabla k(x)|x|^{-b}|u_{\varepsilon}(t)|^{\sigma_b} \,dx \nonumber  \\ 
&\quad +c(\varepsilon) \int_{|x| \geq \ell_0 / 2}|\nabla u_{\varepsilon}(t)|^{2}\,dx.
\end{align}
From \eqref{Conc4}, we have that $c(\varepsilon)\rightarrow 0$, as  $\varepsilon \rightarrow 0$ (uniformly on $t>0$) and taking $\varepsilon>0$ sufficient small we can absorb the last term in the left hand side of the previous inequality. 

Thus, in view of \eqref{virial3}-\eqref{Grad-term}, we finally have
\begin{align}
	z_{R}''(t) & \leq c E\left[u_{0,\varepsilon}\right]+c(\varepsilon) -c(\varepsilon)z_{0}''(t) \\
	& \quad +(c-c(\varepsilon))\int_{|x| \leq \ell_0} x \cdot \nabla k(x)| x|^{-b}\left|u_{\varepsilon}(t)\right|^{\sigma_b}\,dx \\
	& \leq c E\left[u_{0,\varepsilon}\right]+c(\varepsilon)-c(\varepsilon) z_{0}''(t)\\
	&\quad +\frac{c}{2} \int_{|x| \leq \ell_0} x \cdot \nabla k(x)|x|^{-b}\left|u_{\varepsilon}(t)\right|^{\sigma_b} \,d x, 
\end{align}
since we can take $\varepsilon>0$ sufficiently small such that $c(\varepsilon)\ll 1$. 

Therefore, from the local assumption \eqref{LocCond2}, we can discard the integral in the right hand side of the last inequality, to deduce, for all $\varepsilon>0$ sufficiently small, that
$$
z''_R(t)+c(\varepsilon) z_{0}''(t) \leq c E\left[u_{0,\varepsilon}\right]+c(\varepsilon)
$$
and since $E\left[u_{0,\varepsilon}\right] \leq -1/\varepsilon$ (see Lemma \ref{Lem1}) we conclude the desired result by standard arguments.
\end{proof}

\section{Minimal mass blow-up solution}\label{MMBS}
We start this section with the proof of our non-existence result.
\begin{proof}[Proof of Theorem \ref{Thm3}]
We first observe that
\begin{align}\label{Eu0}
E[u_0]&\geq E_{k(0)}[u(t)]+c\int (k(0)-k(x))|x|^{-b}|u(t)|^{\sigma_b}\, dx\nonumber \\
&\geq c\int (k(0)-k(x))|x|^{-b}|u(t)|^{\sigma_b}\,dx,
\end{align}
where in the last inequality we have used \eqref{Ek-est} and $\|u(t)\|_{L^2}=\|Q_{k(0)}\|_{L^2}$.

Assume that there exists a minimal mass blow-up solution with finite maximal existence time $T>0$, then Corollary \ref{Cor4} implies that
\begin{equation}\label{EstCor45}
\lambda(t)^{N / 2} e^{i \gamma(t)} u\left(\lambda(t) x, t\right) \rightarrow Q_{k(0)}(x)\quad  \text { in }\quad  H^{1}  \quad  \text { as }\quad t\to T,
\end{equation}
where $\lambda(t)=\|\nabla Q_{k(0)}\|_{L^2}/\|\nabla u(t)\|_{L^2} \to 0$ and $\gamma(t) \in \mathbb{R}$.

Next we show that the integral in \eqref{Eu0} outside the origin goes to zero as $t\to T$. Indeed, since $k\in L^{\infty}$, for a fixed $R>0$ we have
\begin{align}\label{xR}
\left|\int_{|x|\geq R} (k(0)-k(x))|x|^{-b}|u(t)|^{\sigma_b}\,dx\right|&\lesssim \int_{|x|\geq R} |x|^{-b}|u(t)|^{\sigma_b}\,dx \nonumber\\
&\lesssim \lambda(t)^{-2}\int_{|x|\geq R/\lambda(t)}|x|^{-b}|\lambda(t)^{N / 2} e^{i \gamma(t)} u(\lambda(t) x, t)|^{\sigma_b} \, dx.
\end{align}
Note that
\begin{align}
\left|\|\lambda(t)^{N / 2} e^{i \gamma(t)} u(\lambda(t) \cdot , t)\|_{L^{\sigma_b}_b(|x|\geq R/\lambda(t))}-\|Q_{k(0)}\|_{L^{\sigma_b}_b(|x|\geq R/\lambda(t))} \right|&\leq \|\lambda(t)^{N / 2} e^{i \gamma(t)} u(\lambda(t) \cdot , t)-Q_{k(0)}\|_{L^{\sigma_b}_b}\\
&\to 0, \quad  \text { as }\quad t\to T,
\end{align}
by \eqref{EstCor45} and the Gagliardo-Nirenberg inequality \eqref{GNS}. Moreover, the decay \eqref{DecayQ0} implies
\begin{align}
\lambda(t)^{-2}\int_{|x|\geq R/\lambda(t)}|x|^{-b}|Q_{k(0)}|^{\sigma_b} \, dx&\lesssim  \, \frac{\lambda(t)^{-2+b}}{R^b}\int_{|x|\geq R/\lambda(t)}e^{-\sigma_b\theta|x|} \, dx\\
&\lesssim \frac{\lambda(t)^{-2+b}}{R^b}e^{-c\frac{R}{\lambda(t)}}
\end{align}
and then the integral in \eqref{xR} goes to zero as $t\to T$.

On the other hand, for $t$ close enough to $T$ such that $\ell_0/\lambda(t)\gg 1$, we have from \eqref{alpha0} and similar computations that
\begin{align}\label{EstEu0}
\int_{|x|<\ell_0} (k(0)-k(x))|x|^{-b}|u(t)|^{\sigma_b}\, dx&\gtrsim \int_{|x|<\ell_0}|x|^{1+\alpha_0}|x|^{-b}|u(t)|^{\sigma_b} \, dx \nonumber \\
&\gtrsim \lambda(t)^{\alpha_0-1}\int_{|x|<\ell_0/\lambda(t)}|x|^{1+\alpha_0}|x|^{-b}|\lambda(t)^{N / 2} e^{i \gamma(t)} u(\lambda(t) x, t)|^{\sigma_b} \, dx \nonumber\\
&\gtrsim \lambda(t)^{\alpha_0-1}\int_{|x|<1}|x|^{-b}|Q_{k(0)}|^{\sigma_b} \, dx.
\end{align}
Therefore, the integral in \eqref{EstEu0} goes to infinity as $t\to T$, since $\lambda(t)\to 0$ and $0<\alpha_0<1$, reaching a contradiction with \eqref{Eu0}.\\
\end{proof}

Next, we turn to the existence of minimal mass blow-up solution under a flatness condition on $k(x)$ around the origin given in \eqref{flat}. Recalling the pseudo-conformal symmetry \eqref{QT}, for a fixed $T>0$, define
\begin{equation}\label{QT2}
Q_{T}\left(x,t\right)=\left(\frac{1}{T-t}\right)^{N / 2} e^{i  /(T-t){-i|x|^{2} / 4(T-t)}} Q_{k(0)}\left(\frac{ x}{T-t}\right).
\end{equation}
Direct computations imply that $\|Q_{T}\|_{L^2}=\|Q_{k(0)}\|_{L^2}$ and $\lim_{t\to T}\|\nabla Q_{T}\|_{L^2}=\infty$, however $Q_{T}$ is not a solution of \eqref{PVI}. Inspired by Merle \cite{Me90}, we will show that $Q_{T}$ is close to a sequence of solutions to \eqref{PVI} and applying compactness arguments we construct a minimal mass blow-up solution that blows up at time $T>0$. To this end we need the following new dispersive estimate in weighted Sobolev spaces.

\begin{lemma}\label{DispEst}
Let $0<b<N$ and $\sigma_{b}=\frac{4-2 b}{N}+2$. Then, for all $t \neq 0$ we have the dispersive estimate
$$
\left\| S(t)\left[|\cdot|^{-b}|f|^{\sigma_{b}-1}\right]\right\|_{L_{b}^{\sigma_{b}}} \lesssim |t|^{-2/\sigma_{b}}\left\| f \right\|_{L_{b}^{\sigma_{b}}}^{\sigma_{b}-1}.
$$
\end{lemma}

\begin{proof}
We present a proof based on Lorentz spaces and we briefly review some theory on this topic. For $1 < p < \infty$ and $1 \leq q \leq \infty$ define 
$$
\displaystyle
\|f\|_{L^{p, q}}= \begin{cases}\left( \int_0^{\infty}\left[t^{\frac{1}{p}} f^{*}(t)\right]^q \,\frac{d t}{t}\right)^{\frac{1}{q}} &  \text { if } \quad 1 \leq q<\infty, \\ \sup _{t>0} t^{\frac{1}{p}} f^{*}(t) &  \text { if } \quad  q=\infty,\end{cases}
$$
where the $f^{*}$ is the decreasing
rearrangement of $f$ defined for $t>0$ by
$$
f^*(t)=\inf \left\{s>0 ; m\left\{x \in \mathbb{R}^N:|f(x)|>s\right\} \leq t\right\}.
$$
$L^{p, q}(\mathbb{R}^N)$ denotes the set of all functions $f:\mathbb{R}^N\to \mathbb{C}$ with finite $\|f\|_{L^{p, q}}$. It is a quasi-Banach spaces (see Grafakos \cite[Theorem 1.4.11]{Gr14}). Moreover, considering
$$
f^{* *}(t)=\frac{1}{t} \int_0^t f^*(s) \, d s, \quad \textrm{for}\quad  t>0,
$$
these spaces become Banach spaces with the norm given by
$$
\displaystyle
|||f|||_{L^{p, q}}= \begin{cases}\left( \int_0^{\infty}\left[t^{\frac{1}{p}} f^{**}(t)\right]^q \, \frac{d t}{t}\right)^{\frac{1}{q}} & , \text { if } \quad 1 \leq q<\infty, \\ \sup _{t>0} t^{\frac{1}{p}} f^{**}(t) & , \text { if } \quad  q=\infty,\end{cases}
$$
(see O'Neil \cite[Definition 2.1., page 136]{ON63}). The quasinorm $\|\cdot\|_{L^{p, q}}$ and the norm $|||\cdot|||_{L^{p, q}}$ are in fact equivalent and satisfy
$$
\|f\|_{L^{p, q}}\leq |||f|||_{L^{p, q}}\leq \frac{p-1}{p}\|f\|_{L^{p, q}}.
$$
Indeed, the first inequality is a direct consequence of $f^{\ast}(t)\leq f^{\ast\ast}(t)$ and the second follows from a Hardy inequality (see Hunt \cite[Theorem (Hardy), page 256]{Hu66}). 

Some important fact about the Lorentz spaces are 
\begin{itemize}
\item $ L^{p} = L^{p, p}$ for $1 < p < \infty$ and $L^{p, q}\subset L^{p, r}$, for $1 \leq q<r \leq \infty$ (see Grafakos \cite[Propositions 1.4.5 and 1.4.10]{Gr14});
\item  $\left\||f|^{a}\right\|_{L^{p,q}}=\|f\|_{L^{a p, a q}}$ (see Grafakos \cite[Remark 1.4.7]{Gr14});
\item If $1/p=1/p_1+1/p_2, 1/q\leq 1/q_1+1/q_2, 1< p, p_1, p_2 < \infty, 1\leq q_1, q_2 \leq \infty$, then
\begin{equation}\label{ONIn}
\left\|fg\right\|_{L^{p,q}}\lesssim \left\|f\right\|_{L^{p_1,q_1}}\left\|g\right\|_{L^{p_2,q_2}}.
\end{equation}
(see O'Neil \cite[Theorem 3.4]{ON63})
\end{itemize}

In Braz e Silva-Ferreira-Villamizar-Roa \cite[Lemma 2.1]{SFR09} the authors proved the following dispersive estimate in Lorentz spaces: if $1\leq q\leq +\infty$, $2<p<\infty$ and $\frac{1}{p}+\frac{1}{p'}=1$, then 
\begin{equation}\label{p5}
	\|S(t)f\|_{L^{p,q}}\lesssim |t|^{-\frac{N}{2}(\frac{1}{p'}-\frac{1}{p})}\|f\|_{L^{p',q}},\quad \textrm{for}\quad  t\neq 0.
\end{equation}

Simple computations imply that $|x|^{-b}\in L^{N/b, \infty}$, with $0<b<N$, which implies, from the inequality \eqref{ONIn}, that the multiplication operator $|\cdot|^{-b}:L^{p, q}\to L^{p_1, q_1}$ is bounded for $1/p_1=1/p+b/N$, $q_1\geq q$. Now, since $\sigma'_{b}<2<\sigma_{b}$ and setting $1/q=1/\sigma_b-b/N\sigma_b<1/2$, we have
\begin{align}\label{Stri}
\displaystyle
\left\||\cdot|^{-b / \sigma_{b}} S(t)\left[|\cdot|^{-b}|f|^{\sigma_{b}-1}\right]\right\|_{L^{\sigma_{b}, \sigma_{b}}}
&\lesssim \|S(t)[|\cdot|^{-b}|f|^{\sigma_{b}-1}]\|_{L^{q, \sigma_{b}^{\prime}}}\nonumber \\
&\lesssim |t|^{-\frac{N}{2}(\frac{1}{q'}-\frac{1}{q})}\||\cdot|^{-b/\sigma_b}[|\cdot|^{-b+b/\sigma_b}|f|^{\sigma_b-1}]\|_{L^{q',\sigma_b'}}\nonumber\\
&\lesssim |t|^{-\frac{N}{2}(\frac{1}{q'}-\frac{1}{q})}\left\||\cdot|^{-b+b/\sigma_b}|f|^{\sigma_b-1}\right\|_{L^{\sigma_b',\sigma_b'}}\nonumber\\
&\lesssim |t|^{-\frac{N}{2}(\frac{1}{q'}-\frac{1}{q})}\left\||\cdot|^{-b/\sigma_b}f\right\|_{L^{\sigma_b'(\sigma_b-1),\sigma_b'(\sigma_b-1)}}^{\sigma_b-1}\nonumber\\
&\lesssim |t|^{-\frac{N}{2}(\frac{1}{q'}-\frac{1}{q})}\left\||\cdot|^{-b/\sigma_b}f\right\|_{L^{\sigma_b}}^{\sigma_b-1}.
\end{align}
Finally
\begin{align}
\frac{N}{2}\left(\frac{1}{q'}-\frac{1}{q}\right)&=\frac{N}{2}\left(1-\frac{2}{q}\right)=\frac{N}{2}\left(1-\frac{2(N-b)}{N\sigma_b}\right)=\frac{2}{\sigma_b},
\end{align}
which concludes the proof.
\end{proof}

An immediate consequence of the previous result is the following

\begin{coro}\label{DispEst2}
Let $0<b<N$ and $\sigma_{b}=\frac{4-2b}{N}+2$. Then for all $t \neq 0$, we have
$$
\left\| S(t)\left[|\cdot|^{-b}|u|^{\sigma_{b}-2}v\right]\right\|_{L_{b}^{\sigma_{b}}} \lesssim |t|^{-2/\sigma_{b}}\left\| u \right\|_{L_{b}^{\sigma_{b}}}^{\sigma_{b}-2}\left\| v \right\|_{L_{b}^{\sigma_{b}}}.
$$
\end{coro}

\begin{proof} Holder's inequality yields
\begin{align}
\left\||\cdot|^{-b+b/\sigma_b}|u|^{\sigma_b-2}v\right\|_{L^{\sigma_b'}}&=\left\||\cdot|^{-\frac{b}{\sigma_b}(\sigma_b-2)}|u|^{\sigma_b-2}|\cdot|^{-{b}/{\sigma_b}}v\right\|_{L^{\sigma_b'}}\\
&\leq \left\||\cdot|^{-\frac{b}{\sigma_b}(\sigma_b-2)}|u|^{\sigma_b-2}\right\|_{L^{\sigma_b/(\sigma_b-2)}} \left\||\cdot|^{-{b}/{\sigma_b}}v\right\|_{L^{\sigma_b}}
\end{align}
and arguing as in the proof of estimate \eqref{Stri} we deduce the desired inequality.
\end{proof} 

Before we proceed to the construction of the minimal mass blow-up solution under a flatness assumption on $k(x)$ close to the origin we present some basic properties of $Q_T$ defined in \eqref{QT2}.

\begin{lemma}\label{L1}
For all $t \in \mathbb{R}$, $R>0$ and $p \geq 1$ we have
\begin{itemize}
\item[(i)] $\left\||\cdot|^{-b /p} Q_{T}(t)\right\|_{L^{p}}=c\left|\frac{1}{T-t}\right|^{\frac{{N\left(p-2\right)+2b}}{2p}},\,$ for some $c>0$;
\item[(ii)] $\displaystyle \int_{|x| \geq R}|x|^{-b}\left|Q_{T}(x, t)\right|^{\sigma_b}\, dx \lesssim R^{-2} e^{-\frac{\theta R}{|T-t|}}$, for some $\theta>0$;
\item[(iii)] $\displaystyle \lim _{t \rightarrow T}\left\|Q_{T}(t)\right\|_{L^{2}(|x|\leq R)}=\|Q_{k(0)}\|_{L^2}$.
\end{itemize}
\end{lemma}
\begin{proof}
From definition \eqref{QT}, direct computations imply
\begin{align}
\displaystyle
\left\| |\cdot|^{-b/p} Q_{T}(t)\right\|_{L^{p}}^{p}&=\int |x|^{-b}\left|\frac{1}{T-t}\right|^{\frac{N_{p}}{2}}\left|Q_{k(0)}\left(\frac{ x}{T-t}\right)\right|^{p}\, dx\\
&=\int\left|\frac{1}{T-t}\right|^{\frac{N_{p}}{2}-N+b}|x|^{-b}\left|Q_{k(0)}(x)\right|^{p} dx\\
&=\left|\frac{1}{T-t}\right|^{N\left(p/2-1\right)+b}\||\cdot|^{-b/p} Q_{k(0)}\|_{L^{p}}^{p},
\end{align}
which implies $(i)$ by the property \eqref{DecayQ0}. To prove $(ii)$, we employ similar calculations to deduce
\begin{align}
\displaystyle
\int_{|x| \geq R}|x|^{-b}\left|Q_{T}\left(x,t\right)\right|^{\sigma_b} \, d x&=\left|\frac{1}{T-t}\right|^2 \int_{|x| \geq R/|T-t|} |x|^{-b}\left|Q_{k(0)}(x)\right|^{\sigma_b}\, dx\\
&\leq cR^{-2}\left|\frac{1}{T-t}\right|^{2-b}\int_{R/|T-t|}^{\infty}e^{-\sigma_b{\theta} r}r^{N-1}dr\\
&\leq cR^{-2}e^{-\sigma_b{\theta} R/4|T-t|}.
\end{align}
Finally, since $\left\|Q_{T}(t)\right\|_{L^2}=\|Q_{k(0)}\|_{L^2}$ and
\begin{align}
\displaystyle
\left\|Q_{T}(t)\right\|^2_{L^2(|x|\geq R)}&=\int_{|x| \geq R} \left|\frac{1}{T-t}\right|^N\left|Q_{k(0)}\left(\frac{ x}{T-t}\right)\right|^{2} \,dx\\
&= \int_{|x|\geq R/|T-t|}|Q_{k(0)}(x)|^2\,dx \to 0, \quad \mbox{as} \quad t\to T,
\end{align}
item $(iii)$ follows.
\end{proof}

Now, given $T>0$ and $\varepsilon>0$ small, define $v_{\varepsilon}(t)$ to be the solution of \eqref{PVI} such that 
$$
v_{\varepsilon}(T-\varepsilon)=Q_{T}(T-\varepsilon).
$$ 
By the integral formulation with initial time $T-\varepsilon$ we have for any existence time $t \leq T-\varepsilon$ (we can obviously assume take $t\geq T/2$)
$$
v_{\varepsilon}(t)=S(t-(T-\varepsilon)) Q_{T}(T-\varepsilon)+i \int_{T-\varepsilon}^{t}S(t-s)\left[k(x)|x|^{-b}| v_{\varepsilon}(s)|^{\sigma_{b}-2} v_{\varepsilon}(s)\right] ds
$$
and also since $Q_{T}(t)$ in a solution of \eqref{INLSc}
$$
Q_{T}(t)=S(t-(T-\varepsilon)) Q_{T}(t-\varepsilon)+i \int_{T-\varepsilon}^{t} S(t-s)\left[k(0)|x|^{-b}\left|Q_{T}(s)\right|^{\sigma_{b}-2} Q_{T}(s)\right]\, d s.
$$

Therefore
\begin{align}
v_{\varepsilon}(t)-Q_{T}(t)&=i \int_{T-\varepsilon}^{t}S(t-s)\left[k(x)|x|^{-b}\left(| v_{\varepsilon}(s)|^{\sigma_{b}-2} v_{\varepsilon}(s)-\left|Q_{T}(s)\right|^{\sigma_{b}-2} Q_{T}(s)\right)\right]\, ds\\
& \quad +i \int_{T-\varepsilon}^{t}S(t-s)\left[(k(x)-k(0))|x|^{-b}\left|Q_{T}(s)\right|^{\sigma_{b}-2} Q_{T}(s)\right]\, ds\\
&\equiv I_1+I_2.
\end{align}

From Lemma \ref{DispEst} and Corollary \ref{DispEst2} we obtain
\begin{align}
\|I_1\|_{L_b^{\sigma_b}}&\leq c\|k\|_{L^{\infty}}\int^{T-\varepsilon}_{t}|t-s|^{-2/\sigma_b}\left(\left\| v_{\varepsilon}(s) \right\|_{L_b^{\sigma_{b}}}^{\sigma_{b}-2}+\left\| Q_{T}(s) \right\|_{L_b^{\sigma_{b}}}^{\sigma_{b}-2}\right)\\
& \quad \quad \times\left\| v_{\varepsilon}-Q_{T}(s) \right\|_{L_b^{\sigma_{b}}}\, ds
\end{align}
and
\begin{align}
\|I_2\|_{L_b^{\sigma_b}}&\leq c\int^{T-\varepsilon}_{t}|t-s|^{-2/\sigma_b}\left\| Q_{T}(s) \right\|_{L_b^{\sigma_{b}}}^{\sigma_{b}-2}\left\|(k(0)-k(x)) Q_{T}(s) \right\|_{L_b^{\sigma_{b}}}\,ds.
\end{align}

Now, from Lemma \ref{L1}-$(i)-(ii)$, we have
\begin{equation}\label{QT-est1}
\left\| Q_{T}(t)\right\|_{L_b^{\sigma_b}}=\frac{c}{|T-t|^{{2}/{\sigma_b}}}
\end{equation}
and for some $\theta >0$ (depending only on $\sigma_b$)
$$
\left\|(k(0)-k(x))Q_{T}(t)\right\|_{L_b^{\sigma_b}}\leq 2\|k\|_{L^{\infty}}\left\|Q_{T}(t)\right\|_{L_b^{\sigma_b}(|x|\geq \ell_0)} \leq c\|k\|_{L^{\infty}} \ell_0^{-2} e^{-{2\theta \ell_0 }/{|T-t|}},
$$
which implies for all existence time $t\in [T/2, T-\varepsilon)$
\begin{align}
\|I_2\|_{L_b^{\sigma_b}}&\leq c\|k\|_{L^{\infty}} \ell_0^{-2} \int^{T-\varepsilon}_{t}|t-s|^{-2/\sigma_b} \left|\frac{1}{T-s}\right|^{\frac{2(\sigma_b-2)}{\sigma_b}} e^{-{2\theta \ell_0 }/{|T-s|}}\,ds\\
&\leq c\|k\|_{L^{\infty}} \ell_0^{-2} e^{-{\theta \ell_0 }/{|T-t|}}\int^{T-\varepsilon}_{t}|t-s|^{-2/\sigma_b}\,ds\\
&\leq c\|k\|_{L^{\infty}} \ell_0^{-2} e^{-{\theta \ell_0 |}/{|T-t|}} \int^{T/2}_{0}|\tau|^{-2/\sigma_b}\,d\tau\\
&\leq c \|k\|_{L^{\infty}}\ell_0^{-2} e^{-{\theta \ell_0 }/{|T-t|}}.
\end{align}

On the other hand, assuming that $\left\| v_{\varepsilon}(t)\right\|_{L_b^{\sigma_{b}}} \leq (c+1)\left|\frac{1}{T-t}\right|^{2 / \sigma_{b}}$ and using \eqref{QT-est1}, we deduce
$$
\|I_1\|_{L_b^{\sigma_b}}\leq (c+1)^{\sigma_b-1} \|k\|_{L^{\infty}} \int_{t}^{T-\varepsilon}|t-s|^{-2/\sigma_b}\left|\frac{1}{T-s}\right|^{\frac{2\left(\sigma_{b}-2\right)}{\sigma_{b}}}\left\|v_{\varepsilon}(s)-Q_{T}(s)\right\|_{L_b^{\sigma_b}} \,ds.
$$
Collecting the previous estimates, we get 
\begin{align}
\|v_{\varepsilon}(t)-Q_{T}(t)\|_{L_b^{\sigma_b}}&\leq (c+1)^{\sigma_b-1} \|k\|_{L^{\infty}} \int_{t}^{T-\varepsilon}|t-s|^{-2/\sigma_b}\left|\frac{1}{T-s}\right|^{\frac{2\left(\sigma_{b}-2\right)}{\sigma_{b}}} \\
&\quad \times \left\|v_{\varepsilon}(s)-Q_{T}(s)\right\|_{L_b^{\sigma_b}} \, ds + c \|k\|_{L^{\infty}} \ell_0^{-2} e^{-{\theta \ell_0 }/{|T-t|}}.
\label{uep-Q}
\end{align}

In order to use a method of a priori estimates we need the following numerical inequality.

\begin{lemma}\label{MerNum}
Let $T>0$, $c_{1}>0$ and $\alpha \in(0,1)$. For $\ell>0$ large enough there exists $\tau=\tau(c_1, \ell, \alpha)>0$ such that for all $t \in[\tau, T)$ 
$$
\int_{t}^{T}|t-s|^{-\alpha}\left|\frac{1}{T-s}\right|^{2(1-\alpha)} e^{-\ell/|T-s|} \, ds \leq c_{1} e^{-\ell/|T-t|}.
$$
\end{lemma}

\begin{proof}
First, setting $z=T-s$ the previous inequality is equivalent to
$$
\int_{0}^{T-t}|z-(T-t)|^{-\alpha}\left|\frac{1}{z}\right|^{2(1-\alpha)} e^{-{\ell}/{z}} \, d z \leq c_{1} e^{-{\ell}/{|T-t|}}. 
$$
Since $\tau \leq t \leq T$ implies $ 0 \leq T-t \leq T-\tau$, the desired estimate is also equivalent to existence of $\tau>0$ such that
$$
\int_{0}^{t}|z-t|^{-\alpha}\left|\frac{1}{z}\right|^{2(1-\alpha)} e^{-\ell / z}\, d z \leq c_{1} e^{-\lambda/t}, \,\,\mbox{ for all} \,\, t \in[0, {\tau}].
$$

The rest of the proof follows the steps as in the proof of Lemma 6 in Merle \cite{Me90}. Indeed, taking $z=tw$ the previous inequality is equivalent to
$$
\int_{0}^{1}(1-w)^{-\alpha}w^{-2(1-\alpha)} e^{-\frac{\lambda(1-w)}{tw}}\, d w \leq c_{1} t^{1-\alpha}.
$$
For $w\in (0,1/2)$, $0<t\leq 1$ and $\ell>1$ we have
$$
\frac{\ell(1-w)}{t w} \geq \frac{\ell}{2 t}+\frac{1}{4 w} 
$$
and therefore
$$
\int_{0}^{1 / 2}(1-w)^{-\alpha} w^{-2(1-\alpha)} e^{-\frac{\ell(1-w)}{t w}} \, d w \leq \left(\int_{0}^{1 / 2} w^{-2(1-\alpha)} e^{-1 / 4 w} d w\right) e^{-\ell / 2 t}\leq ce^{-\ell / 2 t}.
$$
Now $c e^{-\ell / 2 t} \leq c_{1} t^{1-\alpha}$ is equivalent to
\begin{equation}\label{talpha}
t^{\alpha-1} e^{-\ell / 2 t} \leq c_{1} / c.
\end{equation}
For fixed $\alpha \in (0,1)$ and $\ell>0$ the left hand side of the previous inequality goes to zero as $t\to 0$. Therefore, for any given positive number $c_1>0$ there exists $\tau(c_1, \ell, \alpha)>0$ such that the inequality \eqref{talpha} holds for any $t\in [0,\tau]$.

On the other hand, for $w\in (1/2,1)$ we have $1-w\leq (1-w)/w$ and therefore
\begin{align}
\int_{1/2}^{1}(1-w)^{-\alpha}w^{-2(1-\alpha)} e^{-\frac{\ell(1-w)}{tw}}\, d w&\leq c\int_{1/2}^{1}(1-w)^{-\alpha} e^{-{\ell(1-w)}/{t}}\, d w\\
&=c\int_{0}^{1/2}s^{-\alpha} e^{-{\ell s}/{t}}\,d s\\
&=c\int_{0}^{1/2t}r^{-\alpha} e^{-{\ell r}}t^{1-\alpha}\,d r\\
&\leq c t^{1-\alpha}\int_{0}^{\infty}r^{-\alpha} e^{-{\ell r}}\,d r.
\end{align}
Since $\alpha \in (0,1)$ the integral in the last line goes to zero as $\ell \to \infty$ and the proof is completed.
\end{proof}

Now, we are able to prove the following result

\begin{prop}\label{Unifuep}
For $\ell_{0}>0$ large enough, there exist $\varepsilon_{0}, K, c, \theta>0$ such that for $\varepsilon \in \left(0, \varepsilon_{0}\right]$ we have
\begin{itemize}
\item[(i)] $\| v_{\varepsilon}(T-\varepsilon_{0}) \|_{L_b^{\sigma_b}} \leq K$;
\item[(ii)] $\|v_{\varepsilon}(t)-Q_{T}(t)\|_{L_b^{\sigma_b}} \leq c e^{-\theta \ell_{0} / (T-t)}, \quad \mbox{for all} \quad t \in\left[T-\varepsilon_{0}, T-\varepsilon\right]$.
\end{itemize}
\end{prop}
 
\begin{proof}
Let $c, \theta>0$ be as in estimate \eqref{uep-Q}. From Lemma \ref{MerNum} (with $\alpha=2 / \sigma_{b}$ and $c_{1}=1 / 4$), there exist $\ell_{0} \gg 1$ and $\varepsilon_{0} \ll T/{2}$ such that
\begin{equation}\label{eEst0}
(c+1)^{\sigma_b-1} \|k\|_{L^{\infty}} \int_{t}^{T}|t-s|^{-2 / \sigma_{b}}\left|\frac{1}{T-s}\right|^{\frac{2(\sigma_{b}-2)}{\sigma_{b}}} e^{-{\theta \ell_0 }/2|T-s|} \,ds \leq \frac{1}{4} e^{-{\theta \ell_0 }/2|T-t|},
\end{equation}
for all $t \in [T-\varepsilon_{0}, T)$.

On the other hand, reducing $\varepsilon_{0}>0$ if necessary, for any $\varepsilon<\varepsilon_{0}$ and for all $t \in [T-\varepsilon_{0}, T-\varepsilon]$, simple computations imply that
\begin{equation}\label{eEst11}
e^{-{\theta \ell_0 }/|T-t|}\leq \frac{1}{2} e^{-{\theta \ell_0 }/2|T-t|}
\end{equation}
and
\begin{equation}\label{eEst1}
e^{-{\theta \ell_0 }/2|T-t|}\leq 1/|T-t|^{2/\sigma_b}.
\end{equation}

Now, let $\varepsilon<\varepsilon_{0}$, recalling that $v_{\varepsilon}(T-\varepsilon)=Q_{T}(T-\varepsilon)$, by the continuity of the flow, there exists $\varepsilon_{1} \in\left(\varepsilon, \varepsilon_{0}\right)$ such that for all  $t \in\left[T-\varepsilon_{1}, T-\varepsilon\right]$ we have
\begin{equation}\label{eEst2}
\|v_{\varepsilon}(t)-Q_{T}(t)\|_{L_b^{\sigma_b}}<e^{-{\theta \ell_0 }/2(T-t)}.
\end{equation}
In view of relation \eqref{QT-est1} and inequalities \eqref{eEst1} and \eqref{eEst2}, we deduce for all  $t \in\left[T-\varepsilon_{1}, T-\varepsilon\right]$ that
\begin{align}
\| v_{\varepsilon}(t)\|_{L_b^{\sigma_b}} & \leq \|v_{\varepsilon}(t)-Q_{T}(t)\|_{L_b^{\sigma_b}}+\|Q_{T}(t)\|_{L_b^{\sigma_b}}\\
& <e^{-{\theta \ell_0 }/2(T-t)}+c /|T-t|^{2 / \sigma_{b}}<(c+1) /|T-t|^{2 / \sigma_{b}}. 
\label{eEst3}
\end{align}
Let $\bar{\varepsilon} \in (\varepsilon_{1}, T/2)$ such that $T-\bar{\varepsilon}$ is the first time where
\begin{equation}\label{eEst4}
\|v_{\varepsilon}(T-\bar{\varepsilon})-Q_{T}(T-\bar{\varepsilon})\|_{L_b^{\sigma_b}}=e^{- \theta \ell_{0}/2\bar{\varepsilon}}.
\end{equation}
If there is no such $\bar{\varepsilon}>0$, then \eqref{eEst2} holds for all $t \in [T-\varepsilon_{0}, T-\varepsilon]$ and estimate $(ii)$ holds. Let us prove the same conclusion in the other case. Indeed, we claim that $\bar{\varepsilon} \geq \varepsilon_{0}$. If $\bar{\varepsilon}<\varepsilon_{0}$, by definition of $\bar{\varepsilon}$ the inequality \eqref{eEst2} holds for all $t \in(T-\bar{\varepsilon}, T-\varepsilon]$ and therefore estimate \eqref{eEst3} holds in the same time interval. Thus, increasing $\ell_0>0$ if necessary such that $c \|k\|_{L^{\infty}} \ell_0^{-2}<1$, we can use \eqref{uep-Q} to deduce
\begin{align}
\|v_{\varepsilon}(T-\bar{\varepsilon})-Q_{T}(T-\bar{\varepsilon})\|_{L_b^{\sigma_b}}&\leq (c+1)^{\sigma_b-1} \|k\|_{L^{\infty}}   \int_{T-\bar{\varepsilon}}^{T-\varepsilon}|(T-\bar{\varepsilon})-s|^{-2/\sigma_b}\left|\frac{1}{T-s}\right|^{\frac{2\left(\sigma_{b}-2\right)}{\sigma_{b}}} \\
&\quad \times \left\|v_{\varepsilon}(s)-Q_{T}(s)\right\|_{L_b^{\sigma_b}}\, ds +  e^{-{\theta \ell_0 }/{\bar{\varepsilon}}}\\
&\leq (c+1)^{\sigma_b-1} \|k\|_{L^{\infty}}   \int_{T-\bar{\varepsilon}}^{T-\varepsilon}|(T-\bar{\varepsilon})-s|^{-2/\sigma_b}\left|\frac{1}{T-s}\right|^{\frac{2\left(\sigma_{b}-2\right)}{\sigma_{b}}}\\
&\quad \times  e^{-{\theta \ell_0 }/2(T-s)} ds +  e^{-{\theta \ell_0 }/{\bar{\varepsilon}}}\\
&\leq \frac{1}{4} e^{-{\theta \ell_0 }/2\bar{\varepsilon}}+ \frac{1}{2} e^{-{\theta \ell_0 }/2\bar{\varepsilon}}<  e^{-{\theta \ell_0 }/2\bar{\varepsilon}},
\end{align}
where in the last line we have used estimates \eqref{eEst0} and \eqref{eEst11}, obtaining a contradiction with \eqref{eEst4}. This completes the proof of $(ii)$, which also implies, in view of \eqref{eEst3}, that 
$$
\| v_{\varepsilon}(T-\varepsilon_{0}) \|_{L_b^{\sigma_b}} \leq (c+1)/\varepsilon_{0}^{2/\sigma_b}, \quad \mbox{for all} \quad \varepsilon \in \left(0, \varepsilon_{0}\right],
$$
finishing the proof.
\end{proof}

Next, we obtain the initial datum for the minimal mass blow-up solution we want to construct.
\begin{lemma}\label{f-phi}
Let $\left\{v_{1/n}\left(T-\varepsilon_{0}\right)\right\}$ for $n \in \mathbb{N}$ given by Proposition \ref{Unifuep}. Then there exists $u_0 \in H^{1}$ such that $v_{1/n}(T-\varepsilon_{0}) \rightarrow u_0$, as $n \rightarrow \infty$, strong in $L_{b}^{\sigma_{b}}$.
\end{lemma}
\begin{proof} 
Since $v_{\varepsilon}(T-\varepsilon)=Q_{T}(T-\varepsilon)$, by mass conservation, we first have
$$
\|v_{\varepsilon}(T-\varepsilon_{0})\|_{L^{2}}=\| v_{\varepsilon}(T-\varepsilon)\|_{L^{2}}=\| Q_{T}(T-\varepsilon)\|_{L^{2}}=\| Q_{k(0)} \|_{L^{2}}.
$$

On the other hand, by energy conservation, we also have
\begin{align}
\|\nabla v_{\varepsilon}(T-\varepsilon_{0})\|^2_{L^{2}}
&=2 E[v_{\varepsilon}(T-\varepsilon_{0})]+\frac{2}{\sigma_b} \int k(x)|x|^{-b}\left|v_{\varepsilon}(T-\varepsilon_{0})\right|^{\sigma_b}\, d x \\
& =2 E\left[Q_{T}(T-\varepsilon)\right]++\frac{2}{\sigma_b} \int k(x)|x|^{-b}\left|v_{\varepsilon}(T-\varepsilon_{0})\right|^{\sigma_b}\, d x\\
&=2 E_{k(0)}\left[Q_{T}(T-\varepsilon)\right]+\frac{2}{\sigma_b} \int (k(0)-k(x))|x|^{-b}\left|Q_{T}(T-\varepsilon)\right|^{\sigma_b}\, d x\\
&\quad +\frac{2}{\sigma_b} \int k(x)|x|^{-b}\left|v_{\varepsilon}(T-\varepsilon_{0})\right|^{\sigma_b}\, d x.
\end{align}

In the last two lines, the first term is independent of $\varepsilon>0$, by energy conservation \eqref{energyk}, the second term goes to zero as $\varepsilon\to 0$ by Lemma \ref{L1}-$(ii)$ and the last term is bounded by Proposition \ref{Unifuep}-$(i)$, since $k\in L^{\infty}$. Therefore the set $\left\{v_{1/n}\left(T-\varepsilon_{0}\right)\right\}_{n\in \mathbb{N}}$ is bounded in $H^1$ and there exists $u_0 \in H^{1}$ such that $v_{1/n}(T-\varepsilon_0) \rightharpoonup u_0$ weakly in $H^{1}$ as $n \rightarrow \infty$. The result follows from the compact embedding $H^1 \hookrightarrow L_{b}^{\sigma_{b}}$.
\end{proof}

We have all the tools to prove the existence of a minimal mass blow-up solution under the flatness assumption \eqref{flat}.
\begin{proof}[Proof of Theorem \ref{MinimalBlow}]
Let $u(t)$ be the solution of \eqref{PVI} with initial datum $u_0$ obtained in Lemma \ref{f-phi}. For any $t\in [T-\varepsilon_0, T)$, $v_{1/n}(\cdot)$ is defined on the time interval $[T-\varepsilon_0, t]$ for large enough $n\in \mathbb{N}$ and the $H^1$ norm is uniformly bound on this time interval. Moreover, since $v_{1/n}(T-\varepsilon_{0}) \rightarrow u_0$ strong in $L_{b}^{\sigma_{b}}$, as $n \rightarrow \infty$, from the well-posedness theory and Proposition \ref{Unifuep}-$(ii)$
$$
\|u(t)-Q_{T}(t)\|_{L_b^{\sigma_b}} \leq c e^{-\theta \ell_{0} / (T-t)}, \quad \mbox{for all} \quad t \in [T-\varepsilon_{0}, T).
$$
This implies, for all $R>0$, that 
\begin{equation}\label{limits}
\|u(t)\|_{L_b^{\sigma_b}(|x|<R)}\to \infty \quad \mbox{and} \quad \|u(t)\|_{L_b^{\sigma_b}(|x|\geq R)}\to 0, \quad \mbox{as} \quad t\to T,
\end{equation}
since $Q_T$ has the same behavior by direct computation (see the proof of Lemma \ref{L1}). We claim that
\begin{equation}\label{limits2}
\int k(x)|x|^{-b}|u(t)|^{\sigma_b}\, d x \to \infty, \quad \mbox{as} \quad t\to T.
\end{equation}
Indeed, in view of assumption \eqref{flat}, we have
$$
\int k(x)|x|^{-b}|u(t)|^{\sigma_b}\, d x= k(0)\|u(t)\|_{L_b^{\sigma_b}(|x|<\ell_0)}+\int_{|x|\geq \ell_0} k(x)|x|^{-b}|u(t)|^{\sigma_b}\, d x
$$
and the limits \eqref{limits} imply \eqref{limits2}, since $k\in L^{\infty}$. Now, applying the energy conservation \eqref{energy}, we deduce
$$
\|\nabla u(t)\|_{L^{2}}\to \infty, \quad \mbox{as} \quad t\to T,
$$
and $u(t)$ is a blow-up solution. Moreover since $v_{1/n}(T-\varepsilon_0) \rightharpoonup u_0$ weakly in $H^{1}$ we have, by mass conservation, that
\begin{align}
\|u_0\|_{L^{2}}
&\leq \liminf_n \|v_{1/n}(T-\varepsilon_0)\|_{L^{2}}= \|Q_{k(0)}\|_{L^{2}}.
\end{align}
Therefore, we must have $\|u_0\|_{L^{2}}= \|Q_{k(0)}\|_{L^{2}}$, otherwise $u(t)$ is a global solution in $H^1$ by Theorem \ref{Thm1}, concluding the proof when $\ell_0>0$ is large enough.

%Next, observe that for all $R>0$ we have by Holder's inequality
%\begin{align}
%\left\|u(t)-Q_T(t)\right\|_{L^2(|x|\leq R)} 
%& \leq \||\cdot|^{b/\sigma_b}|\cdot|^{-b/ \sigma_b}\left(u(t)-Q_T(t)\right)\|_{L^{\sigma_b}(|x|\leq R)} \|1\|_{L^{\sigma'_b}(|x|\leq R)}\\
%&\leq R^{b/\sigma_b}R^{1/\sigma'_b}\| u(t)-Q_{T}(t)\|_{L_b^{\sigma_b}}\\
%&\to 0, \quad \mbox{as} \quad t\to T.
%\end{align}
%Therefore, from Lemma \ref{L1}-$(iii)$
%$$
%\|u(t)\|_{L^2(|x| \leq R)}\to \|Q_{k(0)}\|_{L^2},\quad \mbox{as} \quad t\to T. 
%$$
%and since $\|u(t)\|_{L^{2}}= \|Q_{k(0)}\|_{L^{2}}$ we also have
%$$
%\|u(t)\|_{L^2(|x| > R)}\to 0,\quad \mbox{as} \quad t\to T. 
%$$
%concluding the proof when $\ell_0>0$ is large enough.

For general $\ell_0>0$, we proceed as follows. Given $\lambda>0$ define
$$
k_{\lambda}(x)=k(x/\lambda),
$$
then $k_{\lambda}(x)=k(0)$, for all $|x|<\lambda\ell_0$ and taking $\lambda>0$ sufficiently large, from the previous case for a fixed $T>0$, there exists a minimal mass blow-up solution $u^{\lambda}(t)$ defined on the time interval $[0,\lambda^2T)$, solution of the equation
$$
i \partial_{t} u+\Delta u+k_{\lambda}(x)|u|^{\frac{4-2 b}{N}}u=0
$$
%such that $|u^{\lambda}(t)|^{2} \rightharpoonup \|Q_{k(0)}\|_{L^2}^{2} \delta_{0}$, as $t \rightarrow \lambda^2T$. 
Let
$$
u(x,t)=\lambda^{N/2}u^{\lambda}(\lambda x,\lambda^2t).
$$
One can easily check that it is a minimal mass solution of the \eqref{PVI} equation that blows-up in finite time $T>0$.
\end{proof}

\vspace{0.5cm}
\noindent 
\textbf{Acknowledgments.} M. C. was partially supported by Conselho Nacional de Desenvolvimento Cient\'ifico e Tecnol\'ogico - CNPq (project 307099/2023-7). L.G.F. was partially supported by Coordena\c{c}\~ao de Aperfei\c{c}oamento de Pessoal de N\'ivel Superior - CAPES (project 88881.974077/2024-01), Conselho Nacional de Desenvolvimento Cient\'ifico e Tecnol\'ogico - CNPq (project 307323/2023-4) and Funda\c{c}\~ao de Amparo a Pesquisa do Estado de Minas Gerais - FAPEMIG (project PPM-00685-18).

%\bibliography{biblio}
%\bibliographystyle{abbrv}

\newcommand{\Addresses}{{% additional braces for segregating \footnotesize
		\bigskip
		\footnotesize
		
		MYKAEL A. CARDOSO, \textsc{Department of Mathematics, UFPI, Brazil}\par\nopagebreak
		\textit{E-mail address:} \texttt{mykael@ufpi.edu.br}
		
		\medskip
		
		LUIZ G. FARAH, \textsc{Department of Mathematics, UFMG, Brazil}\par\nopagebreak
		\textit{E-mail address:} \texttt{farah@mat.ufmg.br}

}}
\setlength{\parskip}{0pt}
\Addresses

\end{document}